\documentclass[11pt,a4paper]{amsart}
\usepackage[english]{babel} % Language
\usepackage{amsfonts,amscd,amsmath,amssymb,amsthm} % Math and stuff
\usepackage{xparse} % Package for commands with optional arguments
\usepackage{dsfont} % Package for theorems, lemmas, corollaries, proofs and definitions
\usepackage{xcolor} % Package for sev­eral kinds of color tints, shades, tones, and mixes of ar­bi­trary col­ors
\usepackage[utf8]{inputenc} % Package for input in Latex, for it to be UTF-8 (standard format)

\numberwithin{equation}{section}

\newcommand{\Typ}{\mbox{\rm Typ}}
\newcommand{\C}{\mathbb{C}} % symbol for complex numbers
\newtheorem{theorem}{Theorem}[section]{\bf}{\it}
\newtheorem{proposition}[theorem]{Proposition}{\bf}{\it}
\newtheorem{corollary}[theorem]{Corollary}{\bf}{\it}
\newtheorem{lemma}[theorem]{Lemma}{\bf}{\it}
{\bf}{\it}
{\bf}{\it}
{\bf}{\it}
\theoremstyle{definition}
\newtheorem{definition}[theorem]{Definition}%{\it}{\rm}
\newtheorem{remark}[theorem]{Remark}%{\it}{\rm}
\newtheorem{example}[theorem]{Example}%{\it}{\rm}
\theoremstyle{plain}
\newcounter{theoremintro}
\newtheorem{thmintro}[theoremintro]{Theorem}
\theoremstyle{definition}

{\bf}{\it}
{\bf}{\it}
{\bf}{\it}
{\bf}{\it}

\title{Amenability and paradoxicality in semigroups and C$^*$-algebras}

\author[Pere Ara]{Pere Ara$^{1}$}
\address{Department of Mathematics, Universitat Aut\`{o}noma de Barcelona,
 08193 Bellaterra (Barcelona), Spain and Barcelona Graduate School of Mathematics (BGSMath).}
\email{para@mat.uab.cat}

\author[Fernando Lled\'{o}]{Fernando Lled\'{o}$^{2}$}
\address{Department of Mathematics, University Carlos~III Madrid,
  Avda.~de la Universidad~30, 28911 Legan\'{e}s (Madrid), Spain
  and Instituto de Ciencias Matem\'{a}ticas - ICMAT (CSIC - UAM - UC3M - UCM).}
\email{flledo@math.uc3m.es}

\author[Diego Mart\'{i}nez]{Diego Mart\'{i}nez$^{3}$}
\address{Department of Mathematics, University Carlos~III Madrid,
  Avda.~de la Universidad~30, 28911 Legan\'{e}s (Madrid), Spain
  and Instituto de Ciencias Matem\'{a}ticas - ICMAT (CSIC - UAM - UC3M - UCM).}
\email{lumartin@math.uc3m.es}

\date{\today}
\thanks{{$^{1}$ } Partially supported by MINECO and European Regional Development Fund, jointly, through the grant MTM2017-83487-P, by 
MINECO through the María de Maeztu Programme for Units of Excellence in R$\&$D (MDM-2014-0445)
and by the Generalitat de Catalunya through the grant 2017-SGR-1725.}
\thanks{{$^{2}$ } Supported by research projects MTM2017-84098-P and Severo Ochoa SEV-2015-0554 of the Spanish Ministry of Economy and Competition
(MINECO), Spain.}
\thanks{{$^{3}$ } Supported by research projects MTM2017-84098-P, Severo Ochoa SEV-2015-0554 and BES-2016-077968 of the Spanish Ministry of Economy and Competition
(MINECO), Spain.}

\begin{document}
  %----------------------------------------------------------------------------------------
  % TITLE, ABSTRACT AND TABLE OF CONTENTS
  %----------------------------------------------------------------------------------------
  \begin{abstract}
    We analyze the dichotomy amenable/paradoxical in the context of (discrete, countable, unital) semigroups and corresponding semigroup rings. 
    We consider also F\o lner type characterizations of amenability and give an example of a semigroup whose semigroup ring is algebraically amenable but has no 
    F\o lner sequence. 
    
    In the context of inverse semigroups $S$ we give a characterization of invariant measures on $S$ (in the sense of Day) in terms of two notions: {\em domain measurability} and {\em localization}. Given a unital representation of $S$ in terms of partial bijections on some set $X$ we define a natural generalization of the uniform Roe algebra of a group, which we denote by $\mathcal{R}_X$. We show that the following notions are then equivalent: (1) $X$ is domain measurable; (2) $X$ is not paradoxical; (3) $X$ satisfies the domain F\o lner condition; (4) there is an algebraically amenable dense *-subalgebra of $\mathcal{R}_X$; (5) $\mathcal{R}_X$ has an amenable trace; (6) $\mathcal{R}_X$ is not properly infinite and (7) $[0]\not=[1]$ in the $K_0$-group of $\mathcal{R}_X$. We also show that any tracial state on $\mathcal{R}_X$ is amenable. Moreover, taking into account the localization condition, we give several C*-algebraic characterizations of the amenability of $X$. Finally, we show that for a certain class of inverse semigroups, the quasidiagonality of $C_r^*\left(X\right)$ implies the amenability of  $X$. The reverse implication (which is a natural generalization of Rosenberg's conjecture to this context) is false.
  \end{abstract}
  \maketitle
  \tableofcontents

  %----------------------------------------------------------------------------------------
  % INTRODUCTION
  %----------------------------------------------------------------------------------------
  \section{Introduction}
    The notion of an amenable group was first introduced by von Neumann in \cite{N29} to explain why the paradoxical decomposition of the unit ball in $\mathbb{R}^n$ (the so-called \textit{Banach-Tarski} paradox) occurs only for dimensions greater than two (see \cite{TW16,P00,R02}). Later, F\o lner provided in \cite{F55} a very useful combinatorial characterization of amenability in terms of nets of finite subsets of the group that are almost invariant under left multiplication. This alternative approach was then used to study amenability in the context of algebras over arbitrary fields by Gromov \cite[\S 1.11]{G99} and Elek \cite{E03} (see also \cite{B08,CS08,ALLW18} as well as Definition~\ref{def:groups-amenability}~(\ref{item:5})), in operator algebras by Connes \cite{C76,C76vn} (see also Definition~\ref{def:F-alg}~(\ref{F-alg-2})) and in metric spaces by Ceccherini-Silberstein, Grigorchuk and de la Harpe in \cite{CGH99} (see also \cite{ALLW18}).

    A central aspect of the study of amenability in different mathematical structures is the dynamics inherent to the structure. From this point of view, one can divide the analysis depending on 
    whether the action is injective (e.g., in the case of groups) or singular (e.g., in the case of algebras or operator algebras, where the presence of non-zero divisors is generic). 
    Following this division, it usually happens that notions that are equivalent to amenability in the case of groups are no longer equivalent in a more general case. Thus, semigroups and, 
    in particular, inverse semigroups, provide an interesting frame to reconsider central notions in the theory of amenability with relation to groups, groupoids and C$^*$-algebras (see, e.g., the pioneering works \cite{D57,R80,K84} 
    as well as the recent survey by Lawson in \cite{L19} and references therein). In this article we address von Neumann's original dichotomy - amenable versus paradoxical - in the context of semigroups and semigroup rings, 
    and connect this analysis to C$^*$-algebraic structures associated to inverse semigroups. In particular, we define a C$^*$-algebra for an inverse semigroup which generalizes the uniform Roe-algebra of a group, 
    and then study its trace space in relation to the amenability of the original inverse semigroup.

    Amenability of semigroups has been studied since Day's seminal article (see \cite{D57} as well as other classical references \cite{AW67,K77,DN78,N65}). However, the category of semigroups is too broad to obtain 
    classical equivalences like that between amenability, existence of F\o lner sequences, absence of paradoxical decompositions or algebraic amenability of the corresponding semigroup ring. This has led to a variety of 
    approaches that modify classical definitions and introduce new notions, such as {\em strong and weak F\o lner conditions} or {\em fair amenability} to mention only a few \cite{D16,GK17,Y87}. Some other recent results 
    exploring (geo)metrical aspects of discrete semigroups are presented in \cite{F13,GK13}. Furthermore, following the dynamical point of view mentioned above, semigroups are closer to algebras than they are to groups, 
    since the action of an element $s \in S$ on subsets of $S$ can be singular. In the case that $S$ has a zero element, for instance, its action drastically shrinks the size of any subset of $S$ under multiplication. 
    As an illustration of the singular dynamics involved we show in Theorem~\ref{sem_fvspf} that if $S$ has a F\o lner sequence but does {\em not} have a F\o lner sequence exhausting the semigroup 
    (which we call {\em proper F\o lner sequence} in Definition~\ref{def:groups-amenability}~(\ref{item:3})) then $S$ has a finite principal left ideal. 
    This behavior is characteristic of the dynamics given by multiplication in an algebra (cf., \cite[Theorem~3.9]{ALLW18}) and is not present in the context of groups, where one can easily modify a F\o lner sequence 
    of a group to turn it exhausting. 

An alternative approach to understand the dichotomy on a given category is to use operator algebra techniques for a canonical C*-algebra associated to the initial structure. In the special case of groups two important C*-algebras are the reduced group C*-algebra, denoted by $C_r^*\left(G\right)$, and the uniform Roe algebra of a group, which we denote by $\mathcal{R}_G=\ell^\infty(G)\rtimes_r G$, 
    where $G$ acts on $\ell^\infty(G)$ by left translation. Among other things, R\o rdam and Sierakowski establish in \cite{RS12} a relation between paradoxical decompositions of $G$ and properly 
    infinite projections in $\mathcal{R}_G$. Nevertheless, it is not obvious how to associate a C*-algebra to a general semigroup, since the naive approach would be to define the generators of a possible C*-algebra 
    via $V_s\delta_t:=\delta_{st}$ on the Hilbert space $\ell^2(S)$. This, in general, gives unbounded operators due to the singular dynamics involved. Therefore, when we connect our analysis to C*-algebras we will 
    restrict to the class of inverse semigroups, where the dynamics induced by left multiplication are only locally injective, i.e., injective on the corresponding domains. 
    Some general references for inverse semigroups and, also, in relation to C*-algebras are \cite{Au15,F13,H15,L98,M10,P00,P12,S16,V53}. In Theorem~3.19 of \cite{KLLR16}, Kudryavtseva, Lawson, 
    Lenz and Resende prove a Tarski's type alternative where the invariant measure and the paradoxical decomposition restricts to the space of projections $E(S)$ of the inverse semigroup. 
    In the context of groupoids, B\"onicke and Li (see \cite{BL19}) and Rainone and Sims (see \cite{RS19}) establish a sufficient condition on an \'etale groupoid that ensures pure infiniteness of the reduced 
    groupoid C*-algebra in terms of paradoxicality of compact open subsets of the unit space. See also \cite{AR00,ES17,E08} for additional references on the relation between inverse semigroups, C*-algebras and groupoids.

    Recall that an inverse semigroup $S$ is a semigroup such that for every $s \in S$ there is a unique $s^* \in S$ satisfying $s s^* s = s$ and $s^* s s^* = s^*$. 
    We will assume that our semigroups are unital, discrete and countable. In Proposition~\ref{invsem_ame_nch} we will characterize invariant measures in the sense of Day, i.e., 
    finitely additive probability measures satisfying $\mu(s^{-1}A)=\mu(A)$, $s\in S$, $A\subset S$ (where $s^{-1}A$ denotes the preimage of $A$ by $s$), by means of the two following conditions:
    \begin{enumerate}
      \item[(a)] {\it Localization:} $\mu\left(A\right) = \mu\left(A \cap s^*s A\right)$, for any $s\in S$, $A\subset S$.
      \item[(b)] {\it Domain-measurability:} $\mu\left(s^*s A\right) = \mu\left(sA\right)$, for any $s\in S$, $A\subset S$.
    \end{enumerate}
    Fixing a representation $\alpha\colon S\to\mathcal{I}(X) $ of $S$ in terms of partial bijections on some discrete set $X$ one can consider a 
    natural *-representation $V \colon S\to \mathcal{B}(\ell^2(X))$. Define $\mathcal{R}_{X, alg}$ as the *-algebra generated by the family of partial isometries
    $\{V_s\mid s\in S\}$ and $\ell^\infty(X)$. The C*-algebra $\mathcal{R}_X$ is the norm closure of $\mathcal{R}_{X, alg}$. In particular, taking the left regular representation
    $\iota \colon S\to \mathcal I(S)$ we obtain a Roe algebra $\mathcal R_S$, which is a natural
    generalization of the uniform Roe algebra $\mathcal{R}_G$ of a discrete group. Recall that uniform Roe algebras associated to general discrete metric
    spaces are an important class of C*-algebras that naturally encode properties of the metric space, such as amenability, property (A) or lower
    dimensional aspects (see, e.g., \cite[Theorem~4.9]{ALLW18R} or \cite[Theorem~2.2]{LW18}). We will use this strategy to characterize in different ways
    amenability aspects of the inverse semigroup. In this context one can define notions like $S$-domain F\o lner condition and $S$-paradoxical
    decomposition which correspond, in essence, to the usual notions but restricted to the corresponding domains given by the representation $\alpha$. In
    this way, one of the main results in this article is 
    \begin{thmintro}[cf., Theorem~\ref{invsem_th_ame1}]
    \label{thmintro}
      Let $S$ be a countable and discrete inverse semigroup with identity $1 \in S$, and let $\alpha \colon S\to \mathcal I (X)$ be a representation of $S$ on $X$. Then the following conditions are equivalent:
      \begin{enumerate}
        \item  $X$ is $S$-domain measurable.
        \item  $X$ is not $S$-paradoxical.
        \item  $X$ is $S$-domain F\o lner.
        \item  $\mathcal{R}_{X, alg}$ is algebraically amenable.
        \item  $\mathcal{R}_X$ has an amenable trace.
        \item  $\mathcal{R}_X$ is not properly infinite.
        \item  $\left[0\right] \neq \left[1\right]$ in the K$_0$-group of $\mathcal{R}_X$.
      \end{enumerate}
    \end{thmintro}
    Note that this characterization involves the notions corresponding to domain-measurability (see (b) above). We also characterize the full force of amenability of the action
    in Theorem \ref{invsem_th_ame2}, obtaining in particular that it is equivalent to the fact that no projection associated to an idempotent of $S$ is properly infinite in $\mathcal R _X$ (compare with Theorem \ref{thmintro}(6)).
       
    An important step in the proof of the previous theorem is the construction and analysis of a type semigroup $\mathrm{Typ}(\alpha)$ (see Definition~\ref{def:reptypesemigroup}) associated to the representation $\alpha$. 
    Recall that type semigroups have been considered recently in many interesting situations (see, e.g.,~\cite{AE14,R15}).
    
    Moreover, we also show in this section that every tracial state on $\mathcal{R}_X$ is amenable.
    \begin{thmintro}[cf., Theorem~\ref{invsem_thtr}]
      Let $S$ be a countable and discrete inverse semigroup with identity $1 \in S$, and let $\alpha \colon S\to \mathcal I (X)$ be  a representation. A positive linear functional on $\mathcal{R}_X$ is a trace if and only if it is an amenable trace.
    \end{thmintro}
    Given the representation $V \colon S \rightarrow \mathcal{B}\left(\ell^2\left(X\right)\right)$ introduced above, one can also consider the \textit{reduced semigroup C$^*$-algebra}, that is, the C$^*$-algebra $C^*_r\left(X\right)$ generated by $\left\{V_s\right\}_{s \in S}$. In particular we have the following inclusions:
    \begin{equation}
      C_r^*\left(X\right) := C^*\left(\left\{V_s \mid s \in S\right\}\right) \subset \mathcal{R}_X := C^*\left(\left\{V_s \mid s \in S\right\} \cup \ell^{\infty}\left(X\right)\right) \subset \mathcal{B}\left(\ell^2\left(X\right)\right). \nonumber
    \end{equation}
    Lastly, using the theorems above we also prove a generalization to a result by Rosenberg (cf.,~\cite{H87}) in the setting of inverse semigroups.
    \begin{thmintro}[cf., Theorem~\ref{invsem_thm_rosenberg}]
      Let $S$ be a countable and discrete inverse semigroup with identity $1 \in S$ and with a minimal projection. Let $\alpha \colon S\to \mathcal I (X)$ be a representation on some set $X$ and suppose $C^*_r(X)$ is quasidiagonal. Then $X$ is $S$-amenable.
    \end{thmintro}
    
    The structure of the article is as follows. In Section~\ref{sec_pre} we recall different results around the notion of amenability in the context of groups and algebras that partly motivate our analysis. In particular, we introduce the notion of uniform Roe algebra $\mathcal{R}_G$ of a group $G$ and mention in Theorem~\ref{th_ame_chr} a variety of ways in which one may characterize the amenability of $G$ via $\mathcal{R}_G$. In Section~\ref{sec_sem} we focus first on amenability and F\o lner sequences for general semigroups and semigroup rings. We give an example of a semigroup $S$ which has an algebraically amenable group ring $\C S$, but no F\o lner sequence (see Example~\ref{exa:counter}).

    In the final two sections we restrict our analysis to the case of inverse semigroups. In Section~\ref{sec_invsem} we focus on the algebraic (read as non-C*) aspects of amenability in inverse semigroups. In particular we split Day's invariance condition for measures over amenable inverse semigroups $S$ into the two notions (a) and (b) above, and introduce the type semigroup construction. In Section~\ref{sec_invsem_ctr} we present the C*-aspects of our analysis. For example, we introduce the algebra $\mathcal{R}_X$ and prove that all its traces factor through $\ell^{\infty}\left(X\right)$ via a canonical conditional expectation. We finish the article studying the relation between the quasidiagonality of $C_r^*\left(X\right)$ and the $S$-amenability of $X$. We also mention some questions in relation to this problem.

    {\bf Conventions:} We denote by $A \sqcup B$ the disjoint union of two sets $A$ and $B$. Unless otherwise specified, any measure $\mu$ on a set $X$ will be a \textit{finitely additive probability measure}, i.e., 
    $\mu \colon \mathcal{P}(X) \rightarrow \left[0, 1\right]$, where $\mathcal{P}(X)$ is the power set of $X$, satisfies $\mu(X) = 1$ and $\mu\left(A \sqcup B\right) = \mu\left(A\right) + \mu\left(B\right)$, 
    for every $A, B \subset X$. All semigroups $S$ considered will be countable, discrete and with unit $1 \in S$. A representation of an inverse semigroup $S$ on a set $X$ is a unital semigroup 
    homomorphism $\alpha\colon S\to \mathcal{I}(X)$, where $\mathcal{I}(X)$ denotes the inverse semigroup of partial bijections of $X$. We will denote by $\mathcal{B}(\mathcal{H})$ the algebra of bounded 
    linear operators on a complex separable Hilbert space $\mathcal{H}$.
    
    \textbf{Acknowledgements:} We thank an anonymous referee for his helpful remarks on a previous version of the manuscript.

  %----------------------------------------------------------------------------------------
  % PRELIMINARIES: GROUPS AND ROE ALGEBRAS
  %----------------------------------------------------------------------------------------
  \section{Groups and Uniform Roe Algebras}\label{sec_pre}
    This section aims to give a brief summary to some aspects of amenability needed later. We begin with classical notions in the context of groups and relate these with C*-algebraic concepts using the uniform Roe algebra of a group.
    \begin{definition}\label{def:groups-amenability}
      Let $G$ be a countable group and $\mathcal{A}$ a unital $\mathbb{C}$-algebra of countable dimension.
      \begin{enumerate}
        \item $G$ is (left) \textit{amenable} if there exists a (left) invariant measure on $G$, i.e., a finitely additive probability measure $\mu \colon \mathcal{P}(G) \rightarrow \left[0, 1\right]$ such that $\mu\left(g^{-1} A\right) = \mu\left(A\right)$ for all $g \in G$, $A \subset G$.
        \item $G$ satisfies the \textit{F\o lner condition} if for every $\varepsilon > 0$ and finite $\mathcal{F} \subset G$, there is a finite non-empty $F \subset G$ such that $\left|g F \cup F\right| \leq \left(1 + \varepsilon\right) \left|F\right|$, for every $g \in \mathcal{F}$.
        \item\label{item:3} $G$ satisfies the \textit{proper} F\o lner condition if, in addition, the finite set $F$ can be taken to contain any other set $A \subset G$, i.e., for every $\varepsilon > 0, \mathcal{F} \subset G$ and finite $A \subset G$ there is a finite non-empty $F \subset G$ as above such that $A \subset F$.
        \item $G$ is \textit{paradoxical} if there are sets $A_i, B_j \subset G$ and elements $a_i, b_j \in G$ such that
          \begin{align}
            G & = a_1 A_1 \sqcup \dots \sqcup a_n A_n = b_1 B_1 \sqcup \dots \sqcup b_m B_m \nonumber \\
            & \supset A_1 \sqcup \dots \sqcup A_n \sqcup B_1 \sqcup \dots \sqcup B_m. \nonumber
          \end{align}
        \item\label{item:5} $\mathcal{A}$ is \textit{algebraically amenable} if for every $\varepsilon > 0$ and finite $\mathcal{F} \subset \mathcal{A}$ there is a non-zero finite dimensional subspace $W \leq \mathcal{A}$ such that $\text{dim}\left(A W + W\right) \leq \left(1 + \varepsilon\right) \text{dim}\left(W\right)$ for every $A \in \mathcal{F}$.
      \end{enumerate}
    \end{definition}
    
    Along the lines of these notions, but in the C$^*$-scenario, we can define when a C$^*$-algebra $\mathcal{A}$ captures some aspects of \textit{amenability}, the \textit{F\o lner condition} or \textit{paradoxicality}. For additional motivations and results see, e.g.,  \cite{C76,C76vn,B95,B97,AL14} and references therein.
    \begin{definition}\label{def:F-alg}
      Let $\mathcal{A} \subset \mathcal{B}\left(\mathcal{H}\right)$ be a unital C$^*$-algebra of bounded linear operators on a complex separable Hilbert space $\mathcal{H}$. A state on $\mathcal{A}$ is a positive and linear functional on $\mathcal{A}$ with norm one.
      \begin{enumerate}
       \item\label{item:1} A state $\tau$ on $\mathcal{A}$ is called an \textit{amenable trace} 
        if there is a state $\phi$ on $\mathcal{B}(\mathcal{H})$ extending $\tau$, i.e., $\phi|_\mathcal{A}=\tau$, and satisfying
        \[ 
           \phi\left(AT\right) = \phi\left(TA\right)\;,\quad T \in \mathcal{B}\left(\mathcal{H}\right)\;,\;A \in \mathcal{A}\;.
        \]
        The state $\phi$ is called a \textit{hypertrace} for $\mathcal{A}$. The concrete C*-algebra $\mathcal{A}$ is a called a \textit{F\o lner C*-algebra} if it has an amenable trace.
       \item\label{F-alg-2} $\mathcal{A}$ satisfies the \textit{F\o lner condition} if for every $\varepsilon > 0$ and every finite $\mathcal{F} \subset \mathcal{A}$ there is a non-zero finite rank orthogonal projection $P\in\mathcal{B}(\mathcal{H})$ such that $\left|\left|PA - AP\right|\right|_2 \leq \varepsilon \left|\left|P\right|\right|_2$ for every $A \in \mathcal{F}$, where $\|\cdot\|_2$ denotes the Hilbert-Schmidt norm.
       \item A projection $P \in \mathcal{A}$ is \textit{properly infinite} if there are $V, W \in \mathcal{A}$ such that $P = V^*V = W^*W \geq VV^* + WW^*$. Note that, in this case, the range projections $VV^*$ and $WW^*$ are orthogonal.
       The algebra $\mathcal{A}$ is called \textit{properly infinite} when $1 \in \mathcal{A}$ is properly infinite.
      \end{enumerate}
    \end{definition}

    \begin{remark}\label{rem:foelner}
      The class of F\o lner C$^*$-algebras is also known in the literature as \textit{weakly hypertracial} C$^*$-algebras (cf., \cite{B95}). In \cite[Section~4]{AL14} the first and second authors gave an abstract (i.e., representation independent) characterization of this class of algebras in terms of a net of unital completely positive (u.c.p.) maps into matrices which are asymptotically multiplicative in a weaker norm than the operator norm (see also \cite[Theorem~3.8]{ALLW18R}). It can be shown that an abstract C$^*$-algebra is a \textit{F\o lner C$^*$-algebra} if there exists a non-zero representation $\pi \colon \mathcal{A} \rightarrow \mathcal{B}\left(\mathcal{H}\right)$ such that $\pi(\mathcal{A})$ has an amenable trace (cf., \cite[Theorem~4.3]{AL14}). In general, quasidiagonality is a stronger notion than F\o lner (see, e.g., the examples given in the context of general uniform Roe algebras over metric spaces in \cite[Remark~4.14]{ALLW18R}). However, if $\mathcal{A}$ is a unital nuclear C$^*$-algebra, then it is a F\o lner C$^*$-algebra if and only if $\mathcal{A}$ admits a tracial state (see \cite[Proposition~6.3.4]{BO08}). Note this fact implies that every stably finite unital nuclear C$^*$-algebra is in the F\o lner class.
    \end{remark}

    A classical construction relating C$^*$-algebras and groups is given via the so-called \textit{left regular representation}: the unitary representation $\lambda \colon G \rightarrow \mathcal{B}\left(\ell^2(G)\right)$ defined by $\left(\lambda_g(f)\right)(h) := f\left(g^{-1}h\right)$. The \textit{uniform Roe algebra} $\mathcal{R}_G$ of $G$ is the C$^*$-algebra generated by $\left\{\lambda_g\right\}_{g \in G}$ and $\ell^{\infty}(G)$ viewed as multiplication operators in $\ell^2(G)$, that is,
    \begin{equation}
      \mathcal{R}_G := \text{C}^*\Big(\left\{\lambda_g \mid g \in G\right\} \cup \ell^{\infty}(G) \Big) \subset \mathcal{B}\left(\ell^2(G)\right). \nonumber
    \end{equation}
    The following result shows how one can characterize amenability and paradoxicality of the group in terms of C$^*$-properties of the algebra $\mathcal{R}_G$.
    \begin{theorem}\label{th_ame_chr}
      Let $G$ be a countable and discrete group. The following are equivalent:
      \begin{enumerate}
        \item $G$ is amenable. \label{chr:1}
        \item $G$ is not paradoxical.\label{chr:2}
        \item $G$ has a F\o lner sequence.\label{chr:3}
        \item $\mathbb{C} G$ is algebraically amenable.\label{chr:4}
        \item $\mathcal{R}_G$ has an amenable trace (and hence is a F\o lner C$^*$-algebra).\label{chr:5}
        \item $\mathcal{R}_G$ is not properly infinite.\label{chr:6}
        \item $[0]\not=[1]$ in the $K_0$-group of $\mathcal{R}_G$.\label{chr:7}
      \end{enumerate}
    \end{theorem}
    \begin{proof}
      The equivalences (\ref{chr:1}) $\Leftrightarrow$ (\ref{chr:2}) $\Leftrightarrow$ (\ref{chr:3}) are classical (see, e.g., \cite{P12,KL17}). Their equivalence to (\ref{chr:4}) is due to Bartholdi \cite{B08}. To show the equivalences (\ref{chr:1})$\Leftrightarrow$(\ref{chr:5})$\Leftrightarrow$(\ref{chr:6}) recall that the uniform Roe algebra of the group $G$ can be also seen as a reduced crossed product, i.e., $\mathcal{R}_G=\ell^\infty(G)\rtimes_r G$, where the action of $G$ on $\ell^\infty(G)$ is given by left translation of the argument (see, e.g., \cite[Proposition~5.1.3]{BO08}). R\o rdam and Sierakowski show in \cite[Proposition~5.5]{RS12} a direct equivalence between paradoxicality of $G$ and proper infiniteness for this class of crossed products. In fact, they show that $E \subset G$ is paradoxical if and only if the characteristic function $P_E$ is properly infinite in $\ell^\infty(G)\rtimes_r G$ (see also~\cite[Theorem~4.9]{ALLW18R}). 
      Finally, using the reasoning in \cite[Theorem~4.6]{ALLW18R} one can also prove (\ref{chr:1}) $\Leftrightarrow$ (\ref{chr:7}).
    \end{proof}

    \begin{remark}
      For a general study of the relation between F\o lner C$^*$-algebras and crossed products see also \cite{B95,B97}. Moreover, note that the C$^*$-algebra $\mathcal{R}_G$ contains the reduced group C$^*$-algebra 
      \begin{equation}
        \text{C}^*_r(G)=\text{C}^*\left(\left\{\lambda_g \mid g \in G\right\}\right). \nonumber
      \end{equation}
      However the characterization in terms of proper infiniteness in the preceding theorem would not be true if we replaced the former by the latter. In fact, it is well known (see \cite{CF88}) that the reduced C$^*$-algebra of the free group on two generators $\mathbb{F}_2$ has no non-trivial projection and hence is vacuously not properly infinite. However $\mathbb{F}_2$ is indeed paradoxical. Thus observe that a C$^*$-algebraic characterization of amenability via proper infiniteness requires the existence of non-trivial projections in the C$^*$-algebra, and that in $\mathcal{R}_G$ the existence of nontrivial projections is guaranteed by the characteristic functions in $\ell^\infty(G)$.
    \end{remark}

  %----------------------------------------------------------------------------------------
  % SECTION: AMENABILITY IN SEMIGROUPS
  %----------------------------------------------------------------------------------------
  \section{Semigroups}\label{sec_sem}
  We begin next our analysis of amenability in the context of semigroups. We will see that, from a dynamical point of view, semigroups are closer to algebras than to groups. At this level of generality it is not possible to define a natural C*-algebra which provides the variety of characterizations given in Theorem~\ref{th_ame_chr}.

  Recall that a \textit{semigroup} is a non-empty set $S$ equipped with an associative binary operation $\left(s, t\right) \mapsto st$. The notions treated in Section~\ref{sec_pre} do have an analogue in the semigroup scenario, which relies on the \textit{preimage} of a set. Given $s \in S$ and $A \subset S$, the \textit{preimage of $A$ by $s$} is defined by
  \begin{equation}
    s^{-1} A := \left\{t \in S \mid st \in A\right\}. \nonumber
  \end{equation}
  The following definition for semigroups is due to Day~\cite{D57} (see also \cite{F55,N65}). Recall that by \textit{probability measure} we mean \textit{finitely additive probability measure}.
  \begin{definition}\label{def:sem}
    Let $S$ be a semigroup.
    \begin{enumerate}
      \item \label{sem_def_ame} $S$ is \textit{(left) amenable} if there exists an (left) invariant measure on $S$, i.e., a probability measure $\mu \colon \mathcal{P}\left(S\right) \rightarrow \left[0, 1\right]$ 
      such that $\mu\left(s^{-1} A\right) = \mu\left(A\right)$ for every $s \in S$, $A \subset S$.
      \item \label{sem_def_foln} $S$ satisfies the \textit{(left) F\o lner condition} if for all $\varepsilon > 0$ and finite $\mathcal{F} \subset S$ there is a finite non-empty $F \subset S$ such that $\left|s F \cup F\right| \leq \left(1 + \varepsilon\right) \left|F\right|$ for every $s \in \mathcal{F}$.
      \item $S$ satisfies the \textit{proper F\o lner condition} if, in addition, the F\o lner set $F$ can be taken to contain any other set $A \subset S$, i.e., for every $\varepsilon > 0$, finite $\mathcal{F} \subset S$ 
      and finite $A \subset S$ there is a finite non-empty $F \subset S$ as in (\ref{sem_def_foln}) 
      that, in addition, satisfies $A \subset F$.
    \end{enumerate}
  \end{definition}
  \begin{remark}
    \begin{enumerate}
      \item For the rest of the text we will just consider left amenability as defined in Definition~\ref{def:sem} and we will omit the prefix \textit{left}.
      \item We mention the F\o lner condition given in (\ref{sem_def_foln}) is equivalent to the existence of a net (a sequence if $S$ is countable) $\left\{F_i\right\}_{i \in I}$ of finite non-empty subsets of $S$ such that $\left|s F_i \setminus F_i\right|/\left|F_i\right| \rightarrow 0$ for all $s \in S$. These conditions will be used indistinctly throughout the text.
    \end{enumerate}
  \end{remark}

  We introduce next a stronger notion than amenability in the context of semigroups.
  \begin{definition}\label{def:Smeas}
    A semigroup $S$ is called \textit{measurable} if there is a probability measure $\mu$ on $S$ such that $\mu\left(sA\right) = \mu\left(A\right)$, $s \in S$, $A \subset S$.
  \end{definition}
  It is a standard result that any measurable semigroup is amenable as well. The reverse implication is false in general, although it holds in some classes of semigroups, e.g., for left cancellative ones (see Sorenson's Ph.D.~thesis~\cite{S67} as well as Klawe~\cite{K77}).
  
  The following proposition justifies why we can assume a semigroup $S$ to be countable and unital, as we will normally do in the following sections. In general, given a possibly non-unital semigroup $S$ we can always consider its unitization $S' := S \sqcup \left\{1\right\}$ and define a multiplication in $S'$ extending that of $S$ so that $1$ behaves as a unit. Moreover, as in the case of groups and algebras, the property of amenability is in essence a countable one, at least for a large class of semigroups (including the inverse). Recall from \cite{GK17} that a semigroup $S$ satisfies the \textit{Klawe condition} whenever $sx = sy$ for $s, x, y \in S$ implies there is some $t \in S$ such that $xt = yt$. As mentioned in \cite{GK17}, the Klawe condition is very general and, in particular, left cancellative as well as inverse semigroups satisfy it. 
  \begin{proposition}\label{sem_prop_ncnt}
    Let $S$ be a semigroup and denote by $S'$ its unitization. Then
    \begin{itemize}
     \item[(i)] $S$ is amenable if and only if $S'$ is amenable.
     \item[(ii)] If any countable subset in $S$ is contained in an amenable countable subsemigroup of $S$, then $S$ is amenable. If, in addition, $S$ satisfies the \textit{Klawe condition}, then the reverse implication is also true.
    \end{itemize}
  \end{proposition}
  \begin{proof}
    (i) The proof directly follows from the definition. Indeed, an invariant measure on $S$ can be extended to an invariant measure on $S'$ defining $\mu\left(\left\{1\right\}\right) = 0$. Conversely, $\left\{1\right\}$ is null for any invariant measure on $S'$, so any invariant measure on $S'$ is also an invariant measure on $S$.
    
    (ii) For the first part, let $\mathcal{A}\left(S\right)$ denote the set of countable and amenable subsemigroups of $S$. Furthermore, for $T \in \mathcal{A}\left(S\right)$ denote by $\mu_T$ an invariant measure on $T \subset S$. We may, without loss of generality, extend it to $S$ by defining $\mu_T\left(S \setminus T\right) := 0$. Observe that then $\mu_T\left(t^{-1}A\right) = \mu_T\left(A\right)$ for every $t \in T$, $A \subset T$. Consider the measure:
    \begin{equation}
      \mu \colon \mathcal{P}\left(S\right) \rightarrow \left[0, 1\right], \quad A \mapsto \mu\left(A\right) := \lim_{T \in \mathcal{A}\left(S\right)} \mu_T\left(A\right) = \lim_{T \in \mathcal{A}\left(S\right)} \mu_T\left(A \cap T\right), \nonumber
    \end{equation}
    where the limit is taken along a free ultrafilter of $\mathcal{A}\left(S\right)$. It follows from a straightforward computation that $\mu$ is an invariant measure on $S$.
    
    For the second part we follow a similar route to that of \cite[Proposition~3.4]{ALLW18}. Recall from \cite{GK17} that a semigroup satisfying the Klawe condition is amenable if and only if for every $\varepsilon > 0$ and finite $\mathcal{F} \subset S$ there is a $\left(\varepsilon, \mathcal{F}\right)$-F\o lner set $F \subset S$ such that $\left|F\right| = \left|sF\right|$ for every $s \in \mathcal{F}$ (see Theorem~2.6 in \cite{GK17} in relation with the notion of strong F\o lner condition).

    Let $C = \left\{c_n\right\}_{n \in \mathbb{N}} \subset S$ be a countable subset. In order to construct an amenable semigroup $T \supset C$ we define an increasing sequence $\left\{T_n\right\}_{n \in \mathbb{N}}$ of countable subsemigroups of $S$ by:
    \begin{itemize}
      \item $T_0$ is the subsemigroup generated by $C$.
      \item Suppose $T_i = \left\{t_j\right\}_{j \in \mathbb{N}}$ has been defined. By \cite{GK17} and the previous paragraph, for every $k \in \mathbb{N}$ we may find an $\left(1/k, \left\{t_1, \dots, t_k\right\}\right)$-F\o lner set $F_k \subset S$ such that $\left|t_j F_k\right| = \left|F_k\right|$ for every $j = 1, \dots, k$. We thus define the semigroup $T_{i + 1}$ to be the semigroup generated by $T_i \cup F_1 \cup F_2 \cup \dots$.
    \end{itemize}
    Finally, consider the semigroup $T = \cup_{i \in \mathbb{N}} T_i$. It is straightforward to prove that then $T$ is amenable, countable and contains $C$.
  \end{proof}

  \begin{remark}
    As in the case of metric spaces or algebras (see, e.g., \cite[Section~2.1]{ALLW18R} and \cite[Section~4]{AL14}), an amenable semigroup can have non-amenable sub-semigroups. For instance, take $S := \mathbb{F}_2 \sqcup \left\{0\right\}$, where $0 \omega = \omega 0 = 0$ for every $\omega \in \mathbb{F}_2$. 
    This semigroup $S$ is amenable, since it has a $0$ element, but has a non-amenable sub-semigroup. A more striking fact is that {\it amenable groups} may contain non-amenable semigroups. For instance, 
    the group $G_2$ of isometries of $\mathbb R^2$ is a solvable group containing a 
    non-commutative free semigroup (see \cite[Theorem 1.8, Theorem 14.30]{TW16}).  
  \end{remark}

  For the purpose of this article the main difference between a semigroup and a group is the lack of injectivity under left multiplication. This fact, among other things, makes it impossible to define a canonical \textit{regular representation} in the general semigroup case. We recall next some well-known facts.
  \begin{theorem}\label{th_sem}
    Let $S$ be a countable discrete semigroup. Consider the assertions:
    \begin{enumerate}
      \item $S$ is amenable. \label{th_sem_ame}
      \item $S$ has a F\o lner sequence. \label{th_sem_fol}
      \item $\mathbb{C} S$ is algebraically amenable. \label{th_sem_algame}
    \end{enumerate}
    Then (\ref{th_sem_ame}) $\Rightarrow$ (\ref{th_sem_fol}) $\Rightarrow$ (\ref{th_sem_algame}).
  \end{theorem}
  \begin{proof}
    F\o lner proved the implication (\ref{th_sem_ame}) $\Rightarrow$ (\ref{th_sem_fol}) in the case of groups and, later, Frey and Namioka extended the proof for semigroups (cf., \cite{F55,N65}). To show (\ref{th_sem_fol}) $\Rightarrow$ (\ref{th_sem_algame}) choose a F\o lner sequence $\left\{F_n\right\}_{n \in \mathbb{N}}$ for $S$. Then the linear span of these subsets $W_n$ $:=$ $\text{span}\left\{f \mid f \in F_n\right\}$ defines a F\o lner sequence for $\mathbb{C} S$. In fact, note that $\text{dim}\left(W_n\right) = \left|F_n\right|$ and for any $s \in S$ we have
    \begin{equation}
      \frac{\text{dim}\left(s W_n + W_n\right)}{\text{dim}\left(W_n\right)} \leq \frac{\left|s F_n \cup  F_n\right|}{\left|F_n\right|} \xrightarrow{n \rightarrow \infty} 1, \nonumber
    \end{equation}
    which concludes the proof.
  \end{proof}

  We remark that none of the reverse implications in Theorem~\ref{th_sem} hold in general. It is well known that a finite semigroup may be non-amenable and any such semigroup is a counterexample to the implication (\ref{th_sem_fol}) $\Rightarrow$ (\ref{th_sem_ame}), because if $S$ is finite it has a trivial (constant) F\o lner sequence $F_n = S$. A concrete example was first given by Day in \cite{D57}: let $S = \left\{a, b\right\}$, where $ab = aa = a$ and $ba = bb = b$. Note that in this case any invariant measure $\mu$ must satisfy $\mu\left(b^{-1}\left\{a\right\}\right) = \mu\left(a^{-1}\left\{b\right\}\right) = \mu\left(\emptyset\right) = 0$. Therefore, no probability measure on $S$ can be invariant and, hence, $S$ not amenable.

  The following example is a counterexample to the implication (\ref{th_sem_algame}) $\Rightarrow$ (\ref{th_sem_fol}) in Theorem~\ref{th_sem}.
  \begin{example}\label{exa:counter}
    Consider the additive semigroup of natural numbers $\mathbb{N} = \left\{0, 1, 2, \dots \right\}$ and the free semigroup on two generators $\mathbb{F}^+_2 = \left\{a, b, ab, \dots\right\}$, where we assume that the semigroup $\mathbb{F}^+_2$ has no identity. Denote by $\alpha$ the action of $\mathbb{F}^+_2 \curvearrowright \mathbb{N}$ given by $\alpha \colon \mathbb{F}^+_2 \rightarrow \text{End}\left(\mathbb{N}\right)$, $a, b \mapsto \alpha_a\left(n\right) = \alpha_b\left(n\right) = n - 1$ when $n \geq 1$ and $\alpha_a\left(0\right) = \alpha_b\left(0\right) = 0$. 
    
    We claim the semigroup $S := \mathbb{N} \rtimes_{\alpha} \mathbb{F}^+_2$ does not satisfy the F\o lner condition, while its complex group algebra is algebraically amenable. Note that the element 
    \begin{equation}
      s = \left(0, a\right) - \left(1, a\right) \in \mathbb{C} S \nonumber
    \end{equation}
    clearly satisfies $\left(n, \omega\right) s = 0$ for every $\left(n, \omega\right) \in S$. Therefore $W := \mathbb{C} s$ is trivially a F\o lner subspace for $\mathbb{C} S$, since it is a one-dimensional left ideal. This proves that $\mathbb{C} S$ is algebraically amenable.

    In order to prove that $S$ does not satisfy the F\o lner condition, we shall prove that for any non-empty finite subset $F \subset S$ either $\left|\left(0, a\right) F \setminus F\right| \geq \left|F\right|/50$ or $\left|\left(0, b\right) F \setminus F\right| \geq \left|F\right|/50$. First observe that $\left|\left(0, a\right) F\right| \geq \left|F\right|/2$, and that equality holds if and only if
    \begin{equation}\label{sem_eq_star}
      F = \left\{\left(0, w_1\right), \left(1, w_1\right), \dots, \left(0, w_k\right), \left(1, w_k\right)\right\} \;\; \text{for some} \; w_i \in \mathbb{F}^+_2, \, i = 1, \dots, k.
    \end{equation}
    Indeed, if $F$ is of this form then clearly $\left|\left(0, a\right)F\right| = \left|F\right|/2$. And, conversely, given $\left(n, u\right) \neq \left(m, v\right)$ one has that $\left(0, a\right) \left(n, u\right) = \left(0, a\right) \left(m, v\right)$ only when $u = v$ and $n = 0, m = 1$ or $n = 1, m = 0$.

    Now suppose $F$ is of the form given in Eq.~(\ref{sem_eq_star}) and satisfies $\left|\left(0, a\right) F \setminus F\right| \leq \left|F\right|/5$. Note that, by the observation in the previous paragraph, $\left|\left(0, a\right) F \setminus F\right| \geq \left|F\right|/2 - N_a$, where $N_a$ is the number of words $w_i$ of $F$ that begin with $a$. Since a word cannot begin with $a$ and with $b$, it follows that the number $N_b$ of words that begin with $b$ satisfies $N_b \leq \left|F\right|/5$. Therefore, again, we conclude that $\left|\left(0, b\right) F \setminus F\right| \geq \left|F\right|/2 - \left|F\right|/5 \geq \left|F\right|/5$, as desired. This proves that no set of the form (\ref{sem_eq_star}) can be F\o lner.

    Finally, given an arbitrary $F \subset S$ we may decompose it into $F = F_* \sqcup F'$, where $F_*$ is of the form (\ref{sem_eq_star}) and $F'$ does not contain pairs of elements of the form $\left(0, \omega\right), \left(1,\omega\right)$, with $\omega \in \mathbb F_2^+$. We have
    \begin{align}
      \left|\left(\left(0, a\right) F \cup \left(0, b\right) F\right) \setminus F\right| \geq  \left|\left(0, a\right) F'\right| + \left|\left(0, a\right) F_*\right| + \left|\left(0, b\right) F'\right| + \left|\left(0, b\right) F_*\right| - \left|F\right| = \left|F'\right|. \nonumber
    \end{align}
    Note that the last equality follows from the fact that $\left|\left(0, a\right)F'\right| = \left|F'\right| = \left|\left(0, b\right) F'\right|$. Therefore, if $F'$ is relatively large when compared to $F$, then $F$ itself will not be F\o lner. Suppose hence that $\left|F'\right| \leq \left|F\right|/25$ and $\left|\left(0, a\right)F \setminus F\right| \leq \left|F\right|/25$. Then we have $|F_*| \geq \left(24/25\right) |F|$. Now observe that 
    \begin{equation}
      (0,a)F_*\setminus F_* = [(0,a)F_*\setminus F] \sqcup [(0,a)F_*\cap F']\subseteq ((0,a)F\setminus F) \cup  F', \nonumber
    \end{equation}
    and so
    \begin{equation}
      |(0,a)F_*\setminus F_*| \le \frac{|F|}{25} + \frac{|F|}{25}  \le \frac{2\cdot 25 \cdot  |F_*|}{25\cdot 24} <\frac{|F_*|}{5}. \nonumber
    \end{equation}
    Since $F_*$ is of the form (\ref{sem_eq_star}), it follows that $\left|\left(0, b\right) F_* \setminus F_*\right| \geq \left|F_*\right|/5$.
    Hence 
    \begin{equation}
      |(0,b)F_*\setminus F_*| \ge \frac{|F_*|}{5}\ge \frac{24|F|}{5\cdot 25}. \nonumber
    \end{equation}
    Finally,
    \begin{eqnarray*} 
      |(0,b)F \setminus F| & \geq & |(0,b)F_*\setminus F | \; = \;   |(0,b)F_*\setminus F_*| - |(0,b)F_* \cap F'| \\
                           & \geq & \frac{24|F|}{5\cdot 25} - \frac{|F|}{25} \ge \frac{|F|}{25}\;.
    \end{eqnarray*}
    It remains to consider what happens when $|F'| \geq |F|/25$. In this case, by the above computation, we get
    \begin{equation}
      2 \, \text{max} \{ |(0,a)F\setminus F | \, ,  |(0,b)F\setminus F| \} \ge | ((0,a)F \cup (0,b)F)\setminus F | \ge |F'| \ge \frac{|F|}{25}, \nonumber
    \end{equation}
    and we deduce that either $|(0,a)F\setminus F|$ or $|(0,b)F\setminus F|$ is greater or equal than $|F|/50$. 

    We conclude that no non-empty finite subset $F \subset S$ can be $(\varepsilon, \{ a,b\})$-invariant for $\varepsilon < 1/50$, which proves that $S$ itself does not satisfy the F\o lner condition.
  \end{example}

  In the following result we establish the difference between the F\o lner condition and the proper F\o lner condition. This result is analogous to \cite[Proposition~2.15]{ALLW18R} (see also \cite[Theorem~3.9]{ALLW18R}). We will use this statement in Proposition~\ref{invsem_fol_notpropfl}. Its proof is inspired by the corresponding result in the algebra setting \cite[Theorem~3.9]{ALLW18R}.
  \begin{theorem}\label{sem_fvspf}
    Let $S$ be a semigroup. Suppose that $S$ satisfies the F\o lner condition but not the proper F\o lner condition. Then there is an element $a \in S$ such that $\left|Sa\right| < \infty$.
  \end{theorem}
  \begin{proof}
    Given $\varepsilon > 0$ and a non-empty finite subset $\mathcal{F} \subset S$ define
    \begin{align}
      \text{F\o l}\left(\varepsilon, \mathcal{F}\right) &:= \left\{ F \subset S \; \mid \; 0 < \left|F\right| < \infty \;\; \text{and} \;\; \max_{s \in \mathcal{F}} \frac{\left|s F \setminus F\right|}{\left|F\right|} \leq \varepsilon \right\}, \nonumber \\
      M_{\varepsilon, \mathcal{F}} & := \sup_{F \in \text{F\o l}\left(\varepsilon, \mathcal{F}\right)} \left|F\right| \in \mathbb{N} \cup \left\{\infty\right\}. \nonumber
    \end{align}
    Since $S$ is \textit{not} properly F\o lner there is a pair $\left(\varepsilon_0, \mathcal{F}_0\right)$ with finite $M_{\varepsilon_0, \mathcal{F}_0}$. Note that the pairs $\left(\varepsilon, \mathcal{F}\right)$ are partially ordered by $\left(\varepsilon_1, \mathcal{F}_1\right) \leq \left(\varepsilon_2, \mathcal{F}_2\right)$ if and only if $\mathcal{F}_1 \subset \mathcal{F}_2$ and $\varepsilon_2 \leq \varepsilon_1$. This partial order induces a partial order on $M_{\varepsilon, \mathcal{F}}$ and thus we may suppose that $\varepsilon_0 M_{\varepsilon_0, \mathcal{F}_0} < 1$. Indeed, simply substitute $\varepsilon_0$ with some $\varepsilon_0' < \min\left\{\varepsilon_0, \,1/M_{\varepsilon_0, \mathcal{F}_0}\right\}$.

    We first claim that for any $\varepsilon \in \left(0, \varepsilon_0\right]$ and $\mathcal{F} \supset \mathcal{F}_0$ we have $\text{F\o l}\left(0, \mathcal{F}\right) = \text{F\o l}\left(\varepsilon, \mathcal{F}\right)$. Indeed, the inclusion $\subset$ is obvious. Moreover, for $F \in \text{F\o l}\left(\varepsilon, \mathcal{F}\right)$ and $s \in \mathcal{F}$ we have
    \begin{equation}
      \left|s F \setminus F\right| \leq \varepsilon \left|F\right| \leq \varepsilon M_{\varepsilon, \mathcal{F}} \leq \varepsilon_0 M_{\varepsilon_0, \mathcal{F}_0} < 1 \nonumber
    \end{equation}
    and hence $\left|s F \setminus F\right| = 0$. Therefore $F \in \text{F\o l}\left(0, \mathcal{F}\right)$. Thus it makes sense to consider the largest F\o lner sets with $\varepsilon = 0$:
    \begin{equation}
      \text{F\o l}_{\text{max}}\left(0, \mathcal{F}\right) := \left\{ F \in \text{F\o l}\left(0, \mathcal{F}\right) \; \mid \; \left|F\right| \geq \left|F'\right| \; \text{for all} \; F' \in \text{F\o l}\left(0, \mathcal{F}\right) \right\}. \nonumber
    \end{equation}
    
    Next we claim that if $\mathcal{F} \subset \mathcal{F}'$ and $F_m \in \text{F\o l}_{\text{max}}\left(0, \mathcal{F}\right), F_m' \in \text{F\o l}_{\text{max}}\left(0, \mathcal{F}'\right)$, then $F_m' \subset F_m$. Indeed, suppose the contrary. Then $\widehat{F} := F_m \cup F_m'$ would be in $\text{F\o l}_{\text{max}}\left(0, \mathcal{F}\right)$ and strictly larger than $F_m$, contradicting the maximality condition in the definition of $\text{F\o l}_{\text{max}}\left(0, \mathcal{F}\right)$. In particular, this means that $\text{F\o l}_{\text{max}}\left(0, \mathcal{F}\right)$ has only one element, for if $F_1, F_2 \in \text{F\o l}_{\text{max}}\left(0, \mathcal{F}\right)$ then $F_1 \subset F_2 \subset F_1$. 
   
    Finally, denote by $F_{\mathcal{F}}$ the element of $\text{F\o l}_{\text{max}}\left(0, \mathcal{F}\right)$ and consider the net $\left\{\left| F_{\mathcal{F}} \right|\right\}_{\mathcal{F} \in \mathcal{J}}$, where $\mathcal{J}:=\{\mathcal{F}\subset S \mid |\mathcal{F}|<\infty \; \text{and} \; \mathcal{F}_0\subset\mathcal{F}\}$. This net is decreasing and contained in $\left[1, \left|F_{\mathcal{F}_0}\right|\right] \cap \mathbb{N}$ and, thus, has a limit, which is attained by some $\mathcal{F}_1$. This means that $s F_{\mathcal{F}_1} \subset F_{\mathcal{F}_1}$ for all $s \in S$. Therefore any $a \in F_{\mathcal{F}_1}$ will meet the requirements of the theorem.
  \end{proof}

  %----------------------------------------------------------------------------------------
  % SECTION: AMENABILITY IN INVERSE SEMIGROUPS
  %----------------------------------------------------------------------------------------
  \section{Inverse semigroups}\label{sec_invsem}
  In the rest of the article we will incorporate into the analysis notions of paradoxical decompositions and the relation to C*-algebras in the category of inverse semigroups, i.e., where one only has a locally injective action. While the rest of the text will be devoted to inverse semigroups, this section focuses only on the \textit{algebraic} (meaning non-C$^*$) properties and Section~\ref{sec_invsem_ctr} will focus on how these properties of $S$ translate into properties of a C*-algebra defined as a generalization of the uniform Roe algebra of a group. 
  
  First we recall the definition of inverse semigroup as well as some important structures and examples.
  \begin{definition}
    An \textit{inverse semigroup} is a semigroup $S$ such that for every $s \in S$ there is a unique $s^* \in S$ satisfying $s s^* s = s$ and $s^* s s^* = s^*$.
  \end{definition}
  \begin{example}
    The most important example of an inverse semigroup is that of the set of partial bijections on a given set $X$, denoted by $\mathcal{I}(X)$. Elements $\left(s, A, B\right) \in \mathcal{I}(X)$ are bijections $s \colon A \rightarrow B$, where $A, B \subset X$. The operation of the semigroup is just the composition of maps where it can be defined. This semigroup contains both a zero element, namely $\left(0, \emptyset, \emptyset\right)$, and a unit, namely $\left(\text{id}, X, X\right)$. Just as the elements of a group $G$ can be thought of as bijections of $G$ on itself by left multiplication, every inverse semigroup $S$ can be thought as contained in $\mathcal{I}\left(S\right)$ via the Wagner-Preston representation (see, e.g., \cite{V53,P12,Au15}). 
  \end{example}
  \begin{remark}
    Given an inverse semigroup $S$, the set $E\left(S\right) = \left\{s^*s \mid s \in S\right\}$ is the set of all \textit{idempotents} (or \textit{projections}) of $S$, i.e., elements satisfying $e = e^2 \in S$. Observe that in an inverse semigroup all idempotents commute and satisfy $e^*=e$ (see~\cite{V53} or~\cite[Theorem 3]{L98}). Moreover, $E\left(S\right)$ has the 
    structure of a meet semi-lattice with respect to the order $e \leq f \iff ef = e$, and $S$ is a group if and only if $E\left(S\right)$ only has a single element (the identity in the group). If one considers $S$ as contained in $\mathcal{I}\left(S\right)$, then an idempotent $e \in S$ will be identified with the identity function $\text{id}_{eS} \colon eS \rightarrow eS$.
      
    Note also that $S$ is unital if and only if $E\left(S\right)$ has a greatest element. We shall assume that all our inverse semigroups are unital with unit denoted by $1$. 
  \end{remark}
  We will show in Theorems~\ref{invsem_th_ame01} and~\ref{invsem_th_ame1} that all the different amenability notions are again related in the inverse semigroup case, but not quite as elegantly intertwined as in groups (see Section~\ref{sec_pre}). Such a conclusion might seem surprising, since it is known that the amenability of an inverse semigroup is closely related to the amenability of its group homomorphic image $G\left(S\right)$, as the following result of Duncan and Namioka in \cite{DN78} shows. In the literature, this fact has led to the opinion that amenability in the inverse semigroup case can be traced back to the group case. Our results later will refine this line of thought.
  \begin{theorem}\label{invsem_th_maxhomimg}
    A countable discrete inverse semigroup $S$ is amenable if and only if the group $\text{G}\left(S\right)$ is amenable, where $\text{G}\left(S\right) = S/\sim$ and $s \sim t$ if and only if $es = et$ for some projection $e \in E\left(S\right)$.
  \end{theorem}

    %----------------------------------------------------------------------------------------
    % SUBSECTION: A CHARACTERIZATION OF INVARIANT MEASURES
    %----------------------------------------------------------------------------------------
    \subsection{A characterization of invariant measures}
    Recall from Definition~\ref{def:sem} that a semigroup $S$ is called amenable if there is an invariant probability measure $\mu \colon \mathcal{P}\left(S\right) \rightarrow \left[0, 1\right]$. One handicap to this definition of amenability is that one loses contact with the notion of paradoxical decomposition, which was, from the beginning, close to amenability. In addition, a non-amenable semigroup does not naturally provide elements $s_i, t_j$ whose regular representation induce properly infinite projections in the C$^*$-algebra $\mathcal{R}_G$, as in the group case.
    Avoiding this drawback will be a critical step in the proof of Theorem~\ref{invsem_th_ame1}. We will present an alternative approach taking into account the domain of the action of the semigroup. In this way we can directly relate the non-amenability of $S$ with the proper infiniteness of the identity of the associated C*-algebra. However, before developing the new approach, we present some basic results for inverse semigroups. In particular, the following lemma, whose proof is elementary, will be very useful in the rest of the text.
    \begin{lemma}\label{invsem_lemma_sets}
      Let $S$ be an inverse semigroup. For any $s \in S$ and $A, B \subset S$ the following relations hold:
      \begin{enumerate}
        \item[(i)] \label{invsem_lemma_sets_1} $s\left(s^{-1} A \cap s^*ss^{-1} A\right) = A \cap ss^* A = ss^{-1} A \subset A \subset s^{-1}s A$.
        \item[(ii)] \label{invsem_lemma_sets_2} $ss^*\left(A \cap ss^*B\right) = A \cap ss^*B$.
        \item[(iii)] \label{invsem_lemma_sets_3} $s^{-1}\left(A \setminus ss^*A\right) = \emptyset$.
      \end{enumerate}
    \end{lemma} 
    \begin{proof}
      The inclusions $ss^{-1} A \subset A \subset s^{-1}s A$ follow directly from the definition and (ii) is straightforward to check. To show $s\left(s^{-1} A \cap s^*ss^{-1} A\right) = A \cap ss^* A $ choose $t \in s^{-1}A \cap s^*ss^{-1}A$. Then $st \in A$ and $t = s^*sq$ for some $q \in S$ with $sq \in A$, hence $st = ss^*sq \in ss^*A$. To show the reverse inclusion consider $t \in A \cap ss^*A$, i.e., $A \ni t = ss^*a$ for some $a \in A$. Then $s^*a\in s^{-1}A$ and $t = ss^*s(s^*a) \in s\left(s^{-1} A \cap s^*ss^{-1} A\right)$. The remaining equalities are proved in a similar vein. 
    \end{proof}
    The following result gives a useful characterization of invariant measures that avoids the use of preimages.
    \begin{proposition}\label{invsem_ame_nch}
      Let $S$ be a countable and discrete inverse semigroup with identity $1 \in S$ and $\mu$ be a probability measure on it. Then the following conditions are equivalent:
      \begin{enumerate}
        \item \label{invsem_ame_nch1} $\mu$ is invariant, i.e., $\mu\left(A\right) = \mu\left(s^{-1} A\right)$ for all $s \in S$, $A \subset S$.
        \item \label{invsem_ame_nch2} $\mu$ satisfies the following conditions for all $s \in S$, $A \subset S$:
        \begin{enumerate}
          \item[(2.a)] $\mu\left(A\right) = \mu\left(A \cap s^*s A\right)$.
          \item[(2.b)] $\mu\left(s^*s A\right) = \mu\left(sA\right)$.
        \end{enumerate}
      \end{enumerate}
    \end{proposition}
    \begin{proof}
      For notational simplicity we show conditions (2.a) and (2.b) interchanging the roles of $s$ and $s^*$. The implication (\ref{invsem_ame_nch1}) $\Rightarrow$ (\ref{invsem_ame_nch2}) follows from two simple observations. First, note that
      \begin{equation}
        \mu\left(A \setminus ss^*A\right) = \mu\left(s^{-1}\left(A \setminus ss^*A\right)\right) = \mu\left(\emptyset\right) = 0. \nonumber
      \end{equation}
      Therefore $\mu\left(A\right) = \mu\left(A \cap ss^*A\right) + \mu\left(A \setminus ss^*A\right) = \mu\left(A \cap ss^*A\right)$, as required. Secondly, observe that $s^*A \subset s^{-1}ss^*A$. Thus
      \begin{align}
        \mu\left(ss^*A\right) = \mu\left(s^{-1}ss^*A\right) & = \mu\left(s^{-1}ss^*A \setminus s^*A\right) + \mu\left(s^*A\right) \nonumber \\
        & = \mu\left(\left(s^*\right)^{-1}\left(s^{-1}ss^*A \setminus s^*A\right)\right) + \mu\left(s^*A\right) = \mu\left(s^*A\right). \nonumber
      \end{align}
      The reverse implication (\ref{invsem_ame_nch2}) $\Rightarrow$ (\ref{invsem_ame_nch1}) follows from (2.a), (2.b) and Lemma~\ref{invsem_lemma_sets}(i). In fact,
      \begin{equation}
        \mu\left(s^{-1} A\right) = \mu\left(s^{-1} A \cap s^*s s^{-1} A\right) = \mu\left(s\left(s^{-1} A \cap s^*s s^{-1} A\right)\right) = \mu\left(A \cap ss^*A\right) = \mu\left(A\right), \nonumber
      \end{equation}
      which proves (\ref{invsem_ame_nch1}).
    \end{proof}
    \begin{remark}
      Observe that this characterization indeed restricts to the usual one in the case of groups since then $s^* = s^{-1}$ and $A \cap s s^{-1} A = A$. Therefore condition (2.a) is empty in the group case.
    \end{remark}
    In the following corollary we combine conditions (2.a) and (2.b) into a single one.
    \begin{corollary}\label{invsem_cor_ame}
      Let $S$ be a countable and discrete inverse semigroup and $\mu$ be a probability measure on it. Then $\mu$ is invariant if and only if $\mu\left(A\right) = \mu\left(s\left(A \cap s^*s A\right)\right)$ for all $s \in S$, $A \subset S$.
    \end{corollary}
    \begin{proof}
      Assume that $\mu$ is invariant, hence satisfies conditions (2.a) and (2.b). Since $A \cap s^*s A \subset s^*s A$ we have 
      \begin{equation} 
        \mu\left(A\right) = \mu\left(A \cap s^*s A\right) 
                          = \mu\left(s^*s\left(A \cap s^*s A\right)\right)
                          = \mu\left(s\left(A \cap s^*s A\right)\right). \nonumber
      \end{equation} 
      To show the reverse implication we prove first condition (2.a) which follows from
      \begin{align}
        \mu\left(A\right) = \mu\left(s\left(A \cap s^*s A\right)\right) 
        & = \mu\Big(s^*\big(s\left(A \cap s^*s A\right) \cap ss^*s\left(A \cap s^*s A\right)\big)\Big) \nonumber \\
        & = \mu\left(s^*s \left(A \cap s^*s A\right)\right) = \mu\left(A \cap s^*s A\right) \, , \nonumber
      \end{align}
      where for the last equation we used Lemma~\ref{invsem_lemma_sets}(ii).
      The condition (2.b) follows directly from $\mu\left(A\right) = \mu\left(s\left(A \cap s^*s A\right)\right)$ just by replacing the set $A$ by $s^*sA$.
    \end{proof}

    %----------------------------------------------------------------------------------------
    % SUBSECTION: DOMAIN-MEASURABLE INVERSE SEMIGROUPS
    %----------------------------------------------------------------------------------------
    \subsection{Domain measurable inverse semigroups}
    The characterization of an invariant measure $\mu$ given in Proposition~\ref{invsem_ame_nch} means that $\mu$ is measurable (see Definition~\ref{def:Smeas}) via $s$ but only when the action of $s$ is restricted to its domain (namely $\mu\left(s^*sA\right) = \mu\left(sA\right)$). In addition, the measure of any set $A$ is localized \textit{within the domain of every $s \in S$} (namely $\mu\left(A\right) = \mu\left(A \cap s^*sA\right)$).

    We will show in Theorem~\ref{invsem_th_ame1} that a necessary and sufficient condition for the C$^*$-algebra $\mathcal{R}_S$ to have an amenable trace is the measurability condition on domains given by $\mu\left(s^*sA\right) = \mu\left(sA\right)$. This fact justifies the next definition.
    \begin{definition}\label{def_invsem_dommeas}
      Let $S$ be a countable and discrete inverse semigroup and $A \subset S$ a subset. Then $A$ is \textit{domain measurable} if there is a measure $\mu \colon \mathcal{P}\left(S\right) \rightarrow \left[0, \infty\right]$ such that the following conditions hold:
      \begin{enumerate}
        \item $\mu\left(A\right) = 1$.
        \item $\mu\left(s^*sB\right) = \mu\left(sB\right)$ for all $s \in S$ and $B \subset S$.
      \end{enumerate}
      We say that $S$ is \textit{domain measurable} when the latter holds for $A = S$ and call the corresponding measures {\em domain measures}.
    \end{definition}
    Domain measurable semigroups can be understood as a possible generalization of the amenable groups. To further specify this idea see Theorem~\ref{invsem_th_ame1} and compare it to Theorem~\ref{th_ame_chr}. 

    \begin{remark}
      Recall from~\cite{D16} that a semigroup is called \textit{fairly amenable} if it has a probability measure $\mu$ such that
      \begin{equation}\label{eq:fair}
       \mu\left(A\right) = \mu\left(sA\right)\quad\text{if~$s$~acts~injectively~on~} A\subset S\;.
      \end{equation}
      Observe that if $\mu$ satisfies Eq.~(\ref{eq:fair}), then it satisfies condition (2.b) in Proposition~\ref{invsem_ame_nch} as well, since $s$ acts injectively on $s^*sA$. Therefore, if $S$ is fairly amenable then it is also domain measurable. However, a domain measure satisfying 
      the condition of domain measurability need not satisfy (\ref{eq:fair}). In fact, consider an inverse  semigroup $S$ with a $0$ element and some other element $s \in S$, $s \neq 0$. Then $S$ is domain measurable since it is amenable with an invariant measure $\mu$ satisfying $\mu(\{0\})=1$. This measure, however, cannot implement fair amenability.
    \end{remark}

    \begin{example}\label{invsem_ex_dommeas2}
      We build a class of non-amenable, domain measurable semigroups. Let 
      $A, N$ be disjoint inverse semigroups, with $A$ amenable and $N$ non-amenable. Consider then the semigroup $S = A \sqcup N$, where 
      $an := n =: na$ for every $a \in A$, $n \in N$. It is routine to show $S$ is an inverse semigroup. Furthermore, we claim that it is non-amenable and domain measurable. Indeed, suppose it is amenable and let $\mu$ be an invariant measure on it. For any $n \in N$, $\mu\left(A\right) = \mu\left(n^{-1} A\right) = \mu\left(\emptyset\right) = 0$ and hence $\mu\left(N\right) = 1$. Therefore $\mu$ would restrict to an invariant mean on $N$, contradicting the hypothesis.

      To prove now that $S$ is domain measurable, just choose an invariant measure $\nu$ on $A$ (that exists since $A$ is amenable) and extend it to $S$ as $\widehat{\nu}\left(A' \sqcup N'\right) = \nu\left(A'\right)$, for any $A' \subset A$ and $N' \subset N$. This measure will satisfy $\widehat{\nu}\left(s^*sB\right) = \widehat{\nu}\left(sB\right)$ for every $s \in S$ and $B \subset S$.
    \end{example}

    %----------------------------------------------------------------------------------------
    % SUBSECTION: REPRESENTATIONS OF INVERSE SEMIGROUPS
    %----------------------------------------------------------------------------------------
    \subsection{Representations of inverse semigroups}
    Following~\cite{L98}, we define a representation of a unital inverse semigroup $S$ on a (discrete) set $X$ as a unital semigroup homomorphism 
    $\alpha \colon S \to \mathcal{I}(X)$. One can check that any action $\theta$ of $S$ on $X$ gives a representation $\alpha$ of $S$ on $X$ by the rule 
    $$\alpha_s = (\theta_s|_{\theta_{s^*s}(X)} \colon \theta_{s^*s}(X) \to \theta _{ss^*}(X)).$$
    Indeed the domain for $\theta_s|_{\theta_{s^*s}(X)} \circ \theta_t|_{\theta_{t^*t}(X)}$ is 
    $$\theta_{t^*} (\theta_{tt^*}(X) \cap \theta_{s^*s}(X)) = \theta_{t^*}(X) \cap \theta_{t^*s^*s}(X) = \theta_{t^*s^*s}(X) = \theta_{(st)^*(st)}(X).$$
    If $\alpha$ is a representation, we denote by $D_{s^*s}$ the domain of $\alpha_s$. Note that $\alpha_s $ is a bijection from $D_{s^*s}$ onto $D_{ss^*}$, with inverse $\alpha_{s^*}$.

    \begin{definition}\label{def:reptypesemigroup}
      Let $\alpha \colon S\to \mathcal I (X)$ be a representation of the inverse semigroup $S$ on a set $X$. We define the type semigroup $\Typ(\alpha)$ 
      as the commutative monoid generated by symbols $[A]$ with $A\in \mathcal P (X)$ and relations
      \begin{enumerate}
        \item $[\emptyset] = 0$.
        \item $[A] = [\alpha_s(A)]$ if $A\subseteq D_{s^*s}$.
        \item $[A\cup B] = [A] + [B]$ if $A\cap B=\emptyset$. 
      \end{enumerate}
    \end{definition}
    This definition is very natural and allows to easily check if a map from $\Typ(\alpha)$ to another semigroup is a homomorphism. We show next that $\Typ (\alpha)$ is indeed isomorphic to a type semigroup which is constructed based on Tarski's original ideas. 

    \begin{definition}\label{def:secondtypesemigroup}
      Let $\alpha \colon  S\to \mathcal{I} (X)$ be a representation of the inverse semigroup $S$ on a set $X$. We say $A, B \subset X$ are \textit{equidecomposable}, and write $A \sim B$, if there are sets $A_i \subset X$ and elements $s_i \in S$, $i=1,\dots,n$, such that $A_i\subseteq D_{s_i^*s_i}$ for $i=1,\dots , n$, and 
      \begin{equation}
        A = A_1 \sqcup \dots \sqcup A_n \;\; \text{and} \;\; \alpha_{s_1} (A_1) \sqcup \dots \sqcup \alpha_{s_n}(A_n) = B. \nonumber
      \end{equation}
    \end{definition}
    It is routine to show that $\sim$ is an equivalence relation. Indeed, note that since $1 \in S$ we have $A \sim A$. Furthermore the relation $\sim$ is clearly symmetric by choosing $B_i := \alpha_{s_i} (A_i)$ and the dynamics $t_i := s_i^*$. Finally, if $A \sim B \sim C$ then there are $A_i, B_j \subset X$ and $s_i, t_j \in S$ such that $A_i\subseteq D_{s_i^*s_i}$, $B_j\subseteq D_{t_j^*t_j}$, and 
    \begin{align*}
      A & = A_1 \sqcup \dots \sqcup A_n \;\; \text{and} \;\; \alpha_{s_1} (A_1) \sqcup \dots \sqcup \alpha_{s_n} (A_n) = B, \nonumber \\
      B & = B_1 \sqcup \dots \sqcup B_m \;\; \text{and} \;\; \alpha_{t_1} (B_1) \sqcup \dots \sqcup \alpha_{t_m} (B_m) = C \nonumber.
    \end{align*}
    In this case the sets $A_{ij} = \alpha_{s_i^*}(\alpha_{s_i}(A_i) \cap B_j)$ and the elements $r_{ij} = t_js_i$ implement the relation $A \sim C$.

    Given a representation $\alpha \colon S\to \mathcal I (X)$, consider the following extensions:
    \begin{itemize}
      \item The semigroup $S \times \text{Perm}\left(\mathbb{N}\right)$, where $\text{Perm}\left(\mathbb{N}\right)$ is the finite permutation group of $\mathbb{N}$, that is the group of permutations moving only a finite number of elements.
      \item A set $A \subset X \times \mathbb{N}$ is called \textit{bounded} if $A \subset X \times F$, 
      with $F \subset \mathbb{N}$ finite. These sets are sometimes expressed as $A_1 \times \left\{i_1\right\} \sqcup \dots \sqcup A_k \times \left\{i_k\right\}$ and, if so, each $A_j$ is called a \textit{level}.
    \end{itemize}
    Then there is an obvious representation of $S \times \text{Perm}$ on $X \times \mathbb{N}$ given coordinate-wise, which will be also denoted by $\alpha$.  Hence, it makes sense to ask when two bounded sets $A, B \subset X \times \mathbb{N}$ are equidecomposable. Define
    \begin{equation}
      \widehat{X} := \left\{A \subset X \times \mathbb{N} \mid A \; \text{is bounded}\right\} / \sim. \nonumber
    \end{equation}
    This set has the natural structure of a commutative monoid with $0 = \emptyset$ and sum defined as follows. Given two bounded sets $A, B \subset X \times \mathbb{N}$, let $k \in \mathbb{N}$ be such that $A \cap B' = \emptyset$, where $B' := \left\{\left(b, n + k\right) \mid \left(b, n\right) \in B\right\}$. Then define $\left[A\right] + \left[B\right] := \left[A \sqcup B'\right]$. One can verify that $+$ is well-defined, associative and commutative. This construction was first done by Tarski in~\cite{T38}, and has been used since then extensively (see, e.g.,~\cite{TW16,R03,R15,P00}). We only need the notation $\widehat{X}$ temporarily and after the proof of the following result we will only use the symbol $\Typ (\alpha)$.
    \begin{proposition}\label{prop:equality-of-typesemigroups}
      Let $\alpha $ be a representation of the unital inverse semigroup $S$ on $X$. Then the  map
      $$\gamma \colon \Typ (\alpha) \longrightarrow \widehat{X}, \qquad \gamma ([A]) = [A\times \{ 1 \}]$$
      is a monoid isomorphism.
    \end{proposition}
    \begin{proof}
      Since $[A\times \{ 1 \} ]= [A \times \{ i \}]$ in $\widehat{X}$, we easily see that this map is well-defined and  surjective. To show it is injective, assume that $\gamma  ( \sum_{i=1}^n [A_i]) = \gamma ( \sum_{j=1}^m [B_j])$ for subsets $A_i,B_j$ of $X$. Then
      $$ A:= \bigsqcup_{i=1}^n A_i\times \{ i \} \sim \bigsqcup_{j=1}^m B_j \times \{ j \} =: B,$$
      and so by definition there are subsets  $W_1,\dots , W_l$ of $X$, and numbers $n_1,\dots ,n_l,m_1,\dots , m_l\in \mathbb N$ and elements $s_1,\dots , s_l\in S$ such that $W_k\subseteq D_{s_k^*s_k}$ for $k = 1, \dots, l$ and 
      $$A= \bigsqcup _ {k=1}^l W_k \times \{ n_k \} , \quad B= \bigsqcup_{k=1}^l \alpha_{s_k}(W_k) \times \{ m_k \} .$$
      It follows that there is a partition $\{ 1,\dots , l \} = \bigsqcup_{i=1}^n I_i $ such that for each $i\in \{ 1\dots , n\}$ we have $A_i= \bigsqcup_{j\in I_i} W_j$. We thus get in $\Typ (\alpha)$
      $$\sum _{i=1}^n [A_i] = \sum_{i=1}^n \sum_{j\in I_i} [W_j] = \sum_{i=1}^n \sum_{j\in I_i} [\alpha_{s_j}(W_j)] = \sum_{j=1}^l [\alpha_{s_j}(W_j)] = \sum_{j=1}^m [B_j],$$
      showing injectivity.
    \end{proof}
    For simplicity we will often denote $\alpha_s\left(x\right) \in X$ by $sx$ and $sA$ will stand for $\alpha_s\left(A\right)$ for any $s \in S$, $x \in X$ and $A \subset X$. 
    Recall that $sx$ is defined only if $x \in D_{s^*s}$. We extend next Definition~\ref{def_invsem_dommeas} above to representations.
    
    \begin{definition}\label{invsem_def_dom}
      Let $\alpha \colon S \to \mathcal I (X)$ be a representation of $S$ and let $A \subset X$ be a subset.
      \begin{enumerate}
        \item \label{invsem_def_dommeas} The set $A$ is \textit{$S$-domain measurable} if there is a measure $\mu \colon \mathcal{P}(X) \rightarrow \left[0, \infty\right]$ satisfying the following conditions:
        \begin{enumerate}
          \item $\mu\left( A \right) = 1$.
          \item $\mu\left( B \right) = \mu\left(sB\right)$ for all $s \in S$ and $B \subset D_{s^*s}$.
        \end{enumerate}
        We say that $X$ is \textit{$S$-domain measurable} when the latter holds for $A = X$.
        \item \label{invsem_def_domfol} The set $A$ is \textit{$S$-domain F\o lner} if there is a sequence $\left\{F_n\right\}_{n \in \mathbb{N}}$ of finite, non-empty subsets of $A$ such that
        \begin{equation}
          \frac{\left|s\left(F_n \cap D_{s^*s}\right) \setminus F_n\right|}{\left|F_n\right|} \xrightarrow{n \rightarrow \infty} 0 \nonumber
        \end{equation}
        for all $s \in S$.
        \item \label{invsem_def_dompar} The set $A$ is \textit{$S$-paradoxical} if there are $A_i, B_j \subset X$ and $s_i, t_j \in S$, $i=1,\dots,n$, $j=1,\dots,m$, such that $A_i\subseteq D_{s_i^*s_i}$, $B_j\subseteq D_{t_j^*t_j}$ and
        \begin{align}
          A & = s_1 A_1 \, \sqcup \, \dots \, \sqcup \, s_n A_n = t_1 B_1 \, \sqcup \, \dots \, \sqcup \, t_mB_m \nonumber \\
          & \supset A_1 \, \sqcup \, \dots \, \sqcup \,  A_n \, \sqcup \, B_1 \, \sqcup \, \dots \, \sqcup \,  B_m. \nonumber
        \end{align}
      \end{enumerate}
    \end{definition}

    \begin{remark}
      \begin{itemize}
        \item[(i)] Note that $S$ is \textit{domain measurable} in the sense of Definition~\ref{def_invsem_dommeas} precisely when $S$ is $S$-domain measurable with respect to the canonical representation $\alpha \colon S\to \mathcal I (S)$.
        \item[(ii)] Note also that $A\subset X$ being paradoxical is the same as saying that $2[A] \le [A]$ in $\Typ (\alpha)$.
      \end{itemize}
    \end{remark}

    Recall that, in a commutative semigroup $\mathcal{S}$, we denote by $n \cdot \beta$ the sum $\beta + \dots + \beta$ of $n$ terms. Also, the only (pre-)order that we use on $\mathcal{S}$ is the so-called {\it algebraic pre-order}, defined by $x \le y$ if and only if $x + z = y$ for some $z\in \mathcal{S}$.
    \begin{lemma}\label{invsem_lemma_type}
      Let $\alpha \colon S \to \mathcal I (X)$ be a representation of $S$, and consider the type semigroup $\Typ (\alpha)$ constructed above. Then the following hold:
      \begin{enumerate}
        \item \label{invsem_lemma_type0} For any bounded sets $A, B \subset X \times \mathbb{N}$ if $A \sim B$, then there exists a bijection $\phi \colon A \rightarrow B$ such that for any $C \subset A$ and $D \subset B$ one has $C \sim \phi\left(C\right)$ and $D \sim \phi^{-1}\left(D\right)$.
       \item \label{invsem_lemma_type3} For any $\left[A\right], \left[B\right] \in \Typ (\alpha)$, if $\left[A\right] \leq \left[B\right]$ and $\left[B\right] \leq \left[A\right]$, then $\left[A\right] = \left[B\right]$.
       \item \label{invsem_lemma_type2} A subset $A$ of $X$ is $S$-paradoxical if and only if $\, \left[A \right] = 2 \cdot \left[A \right]$.
       \item \label{invsem_lemma_type4} For any $\left[A\right], \left[B\right] \in \Typ (\alpha) $ and $n \in \mathbb{N}$, if $n \cdot \left[A\right] = n \cdot \left[B\right]$, then $\left[A\right] = \left[B\right]$.
        \item \label{invsem_lemma_type5} If $\left[A\right] \in \Typ (\alpha)$ and $\left(n + 1\right) \cdot \left[A\right] \leq n \cdot \left[A\right]$ for some $n \in \mathbb{N}$, then $\left[A\right] = 2 \cdot \left[A\right]$.
      \end{enumerate}
    \end{lemma}
    \begin{proof}
      The proof of this lemma is virtually the same as in the group case (see, e.g. \cite[p.~10]{R02}). For convenience of the reader we include a sketch of the proofs.

      To construct the bijection $\phi \colon A \rightarrow B$ in (\ref{invsem_lemma_type0}) just define it by multiplication by $s_i$ in each of the subsets $A_i$, where $A = \sqcup_i A_i$ and $B = \sqcup_i s_iA_i$.
       
      (2) There are $\left[A_0\right], \left[B_0\right] \in \Typ (\alpha)$ such that $\left[A\right] + \left[A_0\right] = \left[B\right]$ and $\left[B\right] + \left[B_0\right] = \left[A\right]$. In this case without loss of generality we can suppose that $A \cap A_0 = \emptyset = B \cap B_0$. Choose $\phi \colon A \sqcup A_0 \rightarrow B$ and $\psi \colon B \sqcup B_0 \rightarrow A$ as in (\ref{invsem_lemma_type0}) and consider
      \begin{equation}
        C_0 := A_0, \; C_{n + 1} := \psi\left(\phi\left(C_n\right)\right) \; \text{and} \; \, C := \cup_{n = 0}^{\infty} C_n. \nonumber
      \end{equation}
      It then follows that $\left(B \sqcup B_0\right) \setminus \phi\left(C\right) = \psi^{-1}\left(A \setminus C\right) = \psi^{-1}\left(A \sqcup A_0 \setminus C\right)$ and hence
      \begin{equation}
        A \sqcup A_0 = \left(A \setminus C\right) \sqcup C \sim \psi^{-1}\left(A \setminus C\right) \sqcup \phi\left(C\right) = \left(B \sqcup B_0 \setminus \phi\left(C\right)\right) \sqcup \phi\left(C\right) = B \sqcup B_0. \nonumber
      \end{equation}
      Therefore $B \sim A \sqcup A_0 \sim B \sqcup B_0 \sim A$.

      Now (3) follows from the definitions and (2). 
        
      Claim (\ref{invsem_lemma_type4}) uses graph theory and follows from K\"onig's Theorem (see \cite[Theorem~0.2.4]{P00}). If $n \cdot \left[A\right] = n \cdot \left[B\right]$ then there are sets $A_i, B_j$ with the following properties:
      \begin{enumerate}
        \item[(a)] $A_1, \dots, A_n$ are pairwise disjoint, just as $B_1, \dots, B_n$.
        \item[(b)] $n \cdot \left[A\right] = \left[A_1\right] + \dots + \left[A_n\right] = \left[B_1\right] + \dots + \left[B_n\right] = n \cdot \left[B\right]$.
        \item[(c)] For every $i = 1, \dots, n$ we have $A_i \sim A$ and $B_i \sim B$.
      \end{enumerate}
      Consider then the bijections $\phi_j \colon A_1 \rightarrow A_j$, $\psi_j \colon B_1 \rightarrow B_j$ and $\chi \colon n \cdot\left[A\right] \rightarrow n \cdot \left[B\right]$ induced by $\sim$, as in (\ref{invsem_lemma_type0}). For $a \in A_1$ denote by $\overline{a}$ the set $\left\{\phi_1\left(a\right), \dots, \phi_n\left(a\right)\right\}$ (and analogously for $b \in B_1$). Consider now the bipartite graph defined by:
      \begin{itemize}
        \item Its sets of vertices are $X = \left\{\overline{a} \mid a \in A_1\right\}$ and $Y = \left\{\overline{b} \mid b \in B_1\right\}$.
        \item The vertices $\overline{a}$ and $\overline{b}$ are joined by an edge if $\chi\left(\phi_j\left(a\right)\right) \in \overline{b}$ for some $j = 1, \dots, n$.
      \end{itemize}
      Then this graph is $n$-regular and, by K\"onig's Theorem, it has a perfect matching $F$. In this case it can be checked that the sets
      \begin{align}
        C_{j, k} &:= \left\{a \in A_1 \mid \exists b \in B_1 \; \text{such that} \; \left(\overline{a}, \overline{b}\right) \in F \; \text{and} \; \chi\left(\phi_j\left(a\right)\right) = \psi_k\left(b\right)\right\}, \nonumber \\
        D_{j, k} &:= \left\{b \in B_1 \mid \exists a \in A_1 \; \text{such that} \; \left(\overline{a}, \overline{b}\right) \in F \; \text{and} \; \chi\left(\phi_j\left(a\right)\right) = \psi_k\left(b\right)\right\}, \nonumber
      \end{align}
      are respectively a partition of $A_1$ and $B_1$. Furthermore $\psi_k^{-1} \circ \chi \circ \phi_j$ is a bijection from $C_{j, k}$ to $D_{j, k}$ implementing the relations $C_{j, k} \sim D_{j, k}$. These, in turn, implement $A \sim A_1 \sim B_1 \sim B$.

      Finally, (\ref{invsem_lemma_type5}) follows from (\ref{invsem_lemma_type3}) and (\ref{invsem_lemma_type4}). Indeed, from the hypothesis
      \begin{equation}
        2 \cdot \left[A\right] + n \cdot \left[A\right] = \left(n + 1\right) \cdot \left[A\right] + \left[A\right] \leq n \cdot \left[A\right] + \left[A\right] = \left(n + 1\right) \cdot \left[A\right] \leq n \cdot \left[A\right]. \nonumber
      \end{equation}
      Iterating this argument we get $2n \cdot \left[A\right] \leq n \cdot \left[A\right]$ and, since the other inequality trivially holds, $n \cdot \left[A\right] = 2n \cdot \left[A\right]$. Applying (\ref{invsem_lemma_type4}) we conclude that $\left[A\right] = 2 \cdot \left[A\right]$.
    \end{proof}
    
    Finally, we next recall one of Tarski's fundamental results \cite{T38} (see also \cite[Theorem~0.2.10]{R02}).
    \begin{theorem}\label{invsem_th_typetarski}
      Let $\left(\mathcal{S}, +\right)$ be a commutative semigroup with neutral element $0$ and let $\epsilon \in \mathcal{S}$. The following are then equivalent:
      \begin{enumerate}
        \item[(i)] $\left(n + 1\right) \cdot \epsilon \not\leq n \cdot \epsilon$ for all $n \in \mathbb{N}$.
        \item[(ii)] There is a semigroup homomorphism $\nu \colon \left(\mathcal{S}, +\right) \rightarrow \left(\left[0, \infty\right], +\right)$ such that $\nu\left(\epsilon\right) = 1$.
      \end{enumerate}
    \end{theorem}

    In order to prove the main result of the section (see Theorem~\ref{invsem_th_ame0}) we need to introduce several actions of the inverse semigroup $S$ on canonical spaces associated with $X$. In particular, 
    we will consider the behavior of domain measures as functionals on $\ell^{\infty}(X)$. Given a representation $\alpha \colon S\to \mathcal I (X)$ we define the action of $S$ on $\ell^{\infty}(X)$ by
    \begin{equation}\label{eq:action-S}
      (s f) (x) :=  \left\{ \begin{array}{lcc}
                      f(s^*x) & \text{if} & x\in D_{ss^*} \\
                      0 & \text{if} & x\notin D_{ss^*}.
                    \end{array} \right. 
    \end{equation}
    The next result establishes an invariance condition in the context of states on $\ell^{\infty}(X)$. 
    \begin{proposition}\label{invsem_lemma_mean}
      Let $\alpha \colon S\to \mathcal I (X)$ be a representation of $S$. If $X$ is domain measurable, with domain measure $\mu$ (cf., Definition~\ref{invsem_def_dom}), then there is a state $m \colon \ell^{\infty}(X) \rightarrow \mathbb{C}$ such that
      \begin{equation}\label{invsem_eq_meandom}
        m\left(s f \right) = m\left(f P_{s^*s}\right), \; \text{for} \; f \in \ell^{\infty}(X) \,,\; s\in S,
      \end{equation}
      where $P_{s^*s}$ denotes the characteristic function of $D_{s^*s} \subset X$.
    \end{proposition}
    \begin{proof}
      For a set $B \subset S$ define $m\left(P_B\right) := \mu\left(B\right)$, where $P_B$ denotes the characteristic function on $B$, and extend the definition by linearity to simple functions and by continuity to all $\ell^{\infty}(X)$. Then $m$ satisfies Eq. (\ref{invsem_eq_meandom}) if and only if it does for any characteristic function $P_B$, and this is a consequence of the domain measurability of $\mu$. Indeed, observe that
      \begin{equation}
        s P_B = P_{s(B \cap D_{s^*s})} \nonumber
      \end{equation}
      and hence, by domain measurability, we obtain
      \begin{equation}
        m\left(s  P_B \right)  = m (P_{s(B \cap D_{s^*s})}) = \mu (s(B \cap D_{s^*s})) = \mu (B \cap D_{s^*s}) = m (P_{B \cap D_{s^*s}})= m\left(P_B P_{s^*s}\right), \nonumber
      \end{equation}
      as claimed.
    \end{proof}

    We next observe that the functional $m$ in the latter proposition can be approximated by functionals of finite support.
    \begin{lemma}\label{invsem_lemma_normapp}
      Let $\alpha \colon S\to \mathcal I (X)$ be a representation of $S$. If $X$ is domain measurable then for every $\varepsilon > 0$ and finite $\mathcal{F} \subset S$ there is a positive $h \in \ell^1(X)$ of norm one such that $\left|\left|s^*h - s^*sh\right|\right|_1 < \varepsilon$ for every $s \in \mathcal{F}$. Furthermore, the support of $h$ is finite.
    \end{lemma}
    \begin{proof}
      We denote by $\Omega$ the set of positive $h \in \ell^1(X)$ of norm one and finite support.
      By Proposition~\ref{invsem_lemma_mean} there is a functional $m \colon \ell^{\infty}(X) \rightarrow \mathbb{C}$ such that $m\left(sf\right) = m\left(s^*sf\right)$ for every $s \in S, f \in \ell^{\infty}(X)$. Since the normal states are weak-* dense in $\left(\ell^{\infty}(X)\right)^*$ 
      there is a net $\left\{h_{\lambda}\right\}_{\lambda \in \Lambda}$ in $\Omega$ such that
      \begin{equation}\label{eq_lemma_normapp}
        \left|\phi_{h_{\lambda}}\left(sf\right) - \phi_{h_{\lambda}}\left(s^*sf\right)\right| = \left|\phi_{s^*h_{\lambda}}\left(f\right) - \phi_{s^*sh_{\lambda}}\left(f\right)\right| \rightarrow 0, 
      \end{equation}
      where $\phi_h\left(f\right) = \sum_{x \in X} h\left(x\right) f\left(x\right)$ and $h \in \Omega, f \in \ell^{\infty}(X)$. In order to transform the latter weak convergence to norm convergence, we shall use a variation of a standard technique (see \cite{D57,N65,C76}). 
      Consider the space $E = \left(\ell^1(X)\right)^S$, which, when equipped with the product topology, is a locally convex linear topological space. Consider the map
      \begin{equation}
        T \colon \ell^1(X) \rightarrow E, \quad h \mapsto T\left(h\right) = \left(s^*h - s^*sh\right)_{s \in S}. \nonumber
      \end{equation}
      Since the weak topology coincides with the product of weak topologies on $E$, it follows from Eq.~(\ref{eq_lemma_normapp}) that $0$ belongs to the weak closure of $T(\Omega)$. 
      Furthermore, since $E$ is locally convex and $T(\Omega)$ is convex, its closure in the weak topology and in the product of the norm topologies are the same. Thus we may suppose that the net $\left\{h_{\lambda}\right\}_{\lambda \in \Lambda}$ actually satisfies that 
      $\left|\left|s^*h_\lambda - s^*sh_\lambda\right|\right|_1 \rightarrow 0$ for all $s \in S$, which completes the proof.
    \end{proof}

    The following lemma is straightforward to check, but we mention it for convenience of the reader.
    \begin{lemma}\label{invsem_lemma_hshape}
      Let $X$ be a set. Any $h \in \ell^1(X)_+$ of norm $1$ and finite support can be written as
      \begin{equation}
        h = \left(\beta_1/\left|A_1\right|\right) P_{A_1} + \dots + \left(\beta_N/\left|A_N\right|\right) P_{A_N} \nonumber
      \end{equation}
      for some finite $A_1 \supset A_2 \supset \dots \supset A_N$, where $\beta_i \geq 0$ and $\sum_{i = 1}^N \beta_i = 1$.
    \end{lemma}
    \begin{proof}
      Let $0 =: a_0 < a_1 < \dots < a_N$ be the distinct values of the function $h$. Then, defining $A_i := \left\{x \in X \mid a_i \leq h\left(x\right)\right\}$ we have that $A_1 \supset A_2 \supset \dots \supset A_N$. Furthermore $h = \sum_{i = 1}^N \gamma_i P_{A_i}$, where $\gamma_i = a_i - a_{i - 1}$ for $i \geq 1$. To conclude the proof put $\beta_i := \gamma_i \left|A_i\right|$, $i = 1, \dots, N$, and note that $\|h\|_1=1$ implies $\sum_{i = 1}^N \beta_i = 1$.
    \end{proof}

    \begin{lemma}\label{invsem_lemma_schi}
      Let $\alpha \colon S \rightarrow \mathcal{I}(X)$ be a representation of $S$ and consider $s \in S$, $A \subset X$. Then $\left(s^* P_A - s^*s P_A\right)\left(x\right) < 0$ if and only if $x \in A \cap D_{s^*s} \setminus s^*\left(A \cap D_{ss^*}\right)$.
    \end{lemma}
    \begin{proof}
    By definition of the action given in Eq.~(\ref{eq:action-S}) we compute
    \begin{equation}
        \left(s^* P_A - s^*s P_A\right) \left(x\right) = \left\{
            \begin{array}{ccc}
              1 & \text{if} & x \in D_{s^*s} \setminus A \; \text{and} \; sx \in A \\
              -1  & \text{if} & x \in A \cap D_{s^*s} \; \text{and} \; sx \not\in A \\
              0 &  & \text{otherwise}.
            \end{array}
          \right. \nonumber
      \end{equation}
      Thus, if $\left(s^* P_A - s^*s P_A\right) \left(x\right) < 0$ then $x \in A \cap D_{s^*s} \setminus s^*\left(A \cap D_{ss^*}\right)$. The other 
      implication is clear.
    \end{proof}

    We can finally establish the main theorem of the section, characterizing the domain measurable representations of an inverse semigroup.
    \begin{theorem}\label{invsem_th_ame0}
      Let $S$ be a countable and discrete inverse semigroup with identity $1 \in S$ and $\alpha \colon S\to \mathcal I (X)$ be a representation of $S$ on $X$. The following are then equivalent:
      \begin{enumerate}
        \item \label{invsem_th_ame0_mea} $X$ is $S$-domain measurable.
        \item \label{invsem_th_ame0_par} $X$ is not $S$-paradoxical.
        \item \label{invsem_th_ame0_fol} $X$ is $S$-domain F\o lner.
      \end{enumerate}
    \end{theorem}
    \begin{proof}
      (\ref{invsem_th_ame0_mea}) $\Rightarrow$ (\ref{invsem_th_ame0_par}). Suppose $X$ is $S$-paradoxical. Then, choosing a domain measure $\mu$ 
       and an $S$-paradoxical decomposition of $X$ we would have
      \begin{align}
        1 = \mu\left( X \right) & \geq \mu\left( A_1\right) + \dots + \mu\left( A_n\right) + \mu\left( B_1\right) + \dots + \mu\left( B_m\right) \nonumber \\
        & = \mu\big(s_1  A_1 \sqcup \dots \sqcup s_n A_n\big) + \mu\big(t_1 B_1 \sqcup \dots \sqcup t_m B_m\big) = \mu\left(X\right) + \mu\left(X\right) = 2, \nonumber
      \end{align}
      which gives a contradiction.

      (\ref{invsem_th_ame0_par}) $\Rightarrow$ (\ref{invsem_th_ame0_mea}). Consider the type semigroup $\Typ (\alpha)$ of the action. As $X$ is not $S$-paradoxical we know that $[X]$ and $2 \cdot \left[ X \right]$ are not equal in $\Typ (\alpha)$. It follows from Lemma~\ref{invsem_lemma_type}(5) that $\left(n + 1\right) \cdot \left[ X\right] \not\leq n \cdot \left[ X \right]$ and hence, by Tarski's Theorem~\ref{invsem_th_typetarski}, there exists a semigroup homomorphism $\nu \colon \Typ (\alpha) \rightarrow \left[0, \infty\right]$ such that $\nu\left(\left[ X \right]\right) = 1$. Then we define $\mu\left(B\right) := \nu\left(\left[B \right]\right)$ which satisfies $\mu\left(X\right) = 1$ and $\mu\left(B\right) = \mu\left(s B\right)$ for every $B \subset D_{s^*s}$, proving that $X$ is $S$-domain measurable.

      (\ref{invsem_th_ame0_fol}) $\Rightarrow$ (\ref{invsem_th_ame0_mea}). Let $\left\{F_n\right\}_{n \in \mathbb{N}}$ be a sequence witnessing the 
      $S$-domain F\o lner property of $X$ and let $\omega$ be a free ultrafilter on $\mathbb{N}$. Consider the measure $\mu$ defined by
      \begin{equation}
        \mu\left(B\right) := \lim_{n \rightarrow \omega} \left|B \cap F_n\right|/\left|F_n\right|. \nonumber
      \end{equation}
      It follows from $\omega$ being an ultrafilter that $\mu$ is a finitely additive measure. Thus
      it remains to prove that $\mu\left(B\right) = \mu\left(s B\right)$ for any $B \subset D_{s^*s}$. Observe first that $s$ acts injectively on $B$. 
      Therefore we have
      \begin{align}
        \left|B \cap F_n\right| = \left|s\left(B \cap F_n\right)\right| & \leq \left|sB \cap s\left(F_n \cap D_{s^*s}\right) \cap F_n\right| + \left|\left(sB \cap s\left(F_n \cap D_{s^*s}\right)\right) \setminus F_n\right| \nonumber \\
        & \leq \left|sB \cap F_n\right| + \left|s\left(F_n \cap D_{s^*s}\right) \setminus F_n\right|, \nonumber
      \end{align}
      and hence, normalizing by $\left|F_n\right|$ and taking ultralimits on both sides, we obtain $\mu\left(B\right) \leq \mu\left(s B\right)$. The other inequality follows from a similar argument, noting that
      \begin{equation}
        \left|s^*\left(sB \cap F_n\right)\right| = \left|sB \cap F_n\right| \nonumber
      \end{equation}
      since $s^*$ acts injectively on $sB$.

      (\ref{invsem_th_ame0_mea}) $\Rightarrow$ (\ref{invsem_th_ame0_fol}). To prove this implication we will refine Namioka's trick (see~\cite[Theorem~3.5]{N65}). By Lemma~\ref{invsem_lemma_normapp} and the fact that $X$ is domain measurable we conclude that for every $\varepsilon > 0$ and finite $\mathcal{F} \subset S$ 
      (which we assume to be symmetric, i.e., $\mathcal{F}=\mathcal{F}^*$) there is a positive function $h \in \ell^1(X)$ of norm $1$ and with finite support such that 
      $\left|\left|s^* h - s^*s h\right|\right|_1 < \varepsilon/\left|\mathcal{F}\right|$ for all $s \in \mathcal{F}$. Moreover, by Lemma~\ref{invsem_lemma_hshape}, we may express the function $h$ as a linear combination
      \begin{equation}
        h = \frac{\beta_1}{\left|A_1\right|} P_{A_1} + \dots + \frac{\beta_N}{\left|A_N\right|} P_{A_N}, \quad \text{where} \; \; A_1 \supset A_2 \supset \dots \supset A_N \; \text{and} \; \sum_{i = 1}^N \beta_i = 1. \nonumber
      \end{equation}

      Consider now the set $B_s := \cup_{i = 1}^N \left( A_i \cap D_{s^*s}\right) \setminus s^*\left(A_i \cap D_{ss^*}\right)$. By Lemma~\ref{invsem_lemma_schi}, the function $s^*h - s^*sh$ is non-negative on $X \setminus B_s$ and hence
      \begin{align}
        \frac{\varepsilon}{\left|\mathcal{F}\right|} > \left|\left|s^*h - s^*sh\right|\right|_1 
          & \geq \sum_{x \in X \setminus B_s} s^*h\left(x\right) - s^*sh \left(x\right) 
                 = \sum_{i = 1}^N \frac{\beta_i}{\left|A_i\right|} \left(\sum_{x \in X \setminus B_s}\Big( s^* P_{A_i}\left(x\right) 
                   - s^*s P_{A_i}\left(x\right)\Big)\right) \nonumber \\
          & \geq \sum_{i = 1}^N \frac{\beta_i}{\left|A_i\right|} \left(\sum_{x \in s^*\left(A_i \cap D_{ss^*}\right) \setminus A_i \cap D_{s^*s}}     \Big(s^*P_{A_i}\left(x\right) - s^*s P_{A_i}\left(x\right)\Big)\right) \nonumber \\
          & = \sum_{i = 1}^N \beta_i \frac{\left|s^*\left(A_i \cap D_{ss^*}\right) \setminus A_i \cap D_{s^*s} \right|}{\left|A_i\right|} 
            = \sum_{i = 1}^N \beta_i \frac{\left|s^*\left(A_i \cap D_{ss^*}\right) \setminus A_i \right|}{\left|A_i\right|}. \label{eq_h_betai}
      \end{align}
      Observe that the last inequality follows from the fact that the sets $A_i$ are nested. Indeed, writing $Z_i = A_i \cap D_{s^*s}$ and $T_i= s^*(A_i \cap D_{ss^*})$ 
      and denoting by $Y^c$ the complement in $X$ of a subset $Y$, 
      we need to show that 
      \begin{equation}
      \label{eq:inclusion-in-meet}
     Z_i^c\cap T_i \subset B_s^c =\cap_{j=1}^N (Z_j^c\cup T_j).
      \end{equation}
      Now, if $i \geq j$ then $A_i \subset A_j$ and so $T_i\subset T_j$ which implies that $Z_i^c\cap T_i \subset Z_j^c\cup T_j$. If $i < j$ then $A_j \subset A_i$ and so $Z_j\subset Z_i$ which implies $Z_i^c\subset Z_j^c$ and so
      $Z_i^c\cap T_i\subset Z_j^c\cup T_j$. This shows \eqref{eq:inclusion-in-meet}. The rest of the proof is similar to \cite{N65}. Denote by $I = \left\{1, \dots, N\right\}$ and consider the measure on $I$ given by 
      $\mu\left(J\right) = \sum_{j \in J} \beta_j$ for every $J \subset I$ and put $\mu\left(\emptyset\right) := 0$. For $s \in \mathcal{F}$ consider the set
      \begin{equation}
        K_s := \left\{i \in I \mid \left|s\left(A_i \cap D_{s^*s}\right) \setminus A_i\right| < \varepsilon \left|A_i\right| \right\}. \nonumber
      \end{equation}
      From Eq.~(\ref{eq_h_betai}) it follows that
      \begin{equation}
        \varepsilon/\left|\mathcal{F}\right| > \sum_{i = 1}^N \beta_i \left|s^*\left(A_i \cap D_{ss^*}\right) \setminus A_i \right|/\left|A_i\right| 
                                             \geq \varepsilon \sum_{i \in I \setminus K_{s^*}} \beta_i = \varepsilon \, \mu\left(I \setminus K_{s^*}\right) \nonumber
      \end{equation}
      and, thus, $\mu\left(I \setminus K_{s^*}\right) < 1/\left|\mathcal{F}\right|$. From this and since $\mathcal{F}=\mathcal{F}^*$ we obtain
      \begin{equation}
        1 - \mu\left(\cap_{s \in \mathcal{F}} K_s\right) = \mu\left(I \setminus \cap_{s \in \mathcal{F}} K_s\right) = \mu\left(\cup_{s \in \mathcal{F}} I \setminus K_s\right) \leq \sum_{s \in \mathcal{F}} \mu\left(I \setminus K_s\right) < 1. \nonumber
      \end{equation}
      Therefore the set $\cap_{s \in \mathcal{F}} K_s$ is not empty since its measure is non-zero; for any index $i_0 \in \cap_{s \in \mathcal{F}} K_s$ 
      we will have that the corresponding set $A_{i_0}$ satisfies the domain F\o lner condition.
      \end{proof}

      \begin{remark}
       We mention here that the theory of type semigroups for representations of inverse semigroups includes the corresponding theory for partial actions of groups. 
       Given a (discrete) group $G$ and a non-empty set $X$, Exel defines the notion of a partial action of $G$ on $X$ (see, e.g., \cite{E17}). 
       In this context one can associate in a natural way the type semigroup $\mathrm{Typ} (X,G)$ to the given partial action (see e.g. \cite[Section 7]{AE14}). 
       Moreover, in \cite{Exel98} Exel associates
       to each group $G$ an inverse semigroup $S(G)$ such that the partial actions of $G$ on $X$ are in bijective correspondence with the representations  
       $\alpha\colon S(G)\to \mathcal{I}(X)$. (Note that representations of inverse semigroups are called {\it actions} in \cite{Exel98}.) 
       In this context, it can be shown, using the abstract definitions of these semigroups, that the type semigroup $\mathrm{Typ}(\alpha)$ introduced in 
       Definition~\ref{def:reptypesemigroup} is naturally isomorphic to the type semigroup 
       $\mathrm{Typ} (X,G)$ of the corresponding partial action of $G$ on $X$. 
      \end{remark}
      
    %----------------------------------------------------------------------------------------
    % SUBSECTION: AMENABLE INVERSE SEMIGROUPS
    %----------------------------------------------------------------------------------------
    \subsection{Amenable inverse semigroups}\label{sec_ameinvsem}
    The goal of this section is to prove the analogue of Theorem~\ref{invsem_th_ame0} but considering amenable representations instead of
    the weaker notion of domain measurable ones. 
    Therefore we will have to refine the reasoning of the previous section including the localization condition. In fact,
    let us first recall that by Proposition~\ref{invsem_ame_nch} the classical definition of invariant measure given by Day can be characterized by domain measurability and the condition
    \begin{equation}
      \mu\left(A \cap s^*s A\right) = \mu\left(A\right)\;,\quad s\in S\,,\; A\subset S\; . \nonumber
    \end{equation}
    Note that, since $A \cap s^*sA = A \cap s^*sS$, we can call this property \textit{localization}, for the measure $\mu$ is concentrated in the domain of the projection $s^*s \in E\left(S\right)$. We now extend this definition to the context of representations.
    \begin{definition}
      Let $\alpha \colon S \to \mathcal I (X)$ be a representation of the inverse semigroup $S$ and let $A \subset X$ be a subset. Then $A$ is 
      \textit{$S$-amenable} when there is a measure $\mu \colon \mathcal{P}(X) \rightarrow \left[0, \infty\right]$ such that:
      \begin{enumerate}
        \item $\mu\left( A \right) = 1$.
        \item $\mu\left( B \right) = \mu\left(\alpha _s(B)\right)$ for all $s \in S$ and $B \subset D_{s^*s}$.
        \item $\mu\left( B \right) = \mu\left(B \cap D_{t^*t}\right)$ for all $t \in S$ and $B \subset X$.
      \end{enumerate}
      We say that $X$ is \textit{$S$-amenable} when the latter holds for $A = X$.
    \end{definition}

    The following Lemma is just a simple observation, but it will be useful for later use.
    \begin{lemma}\label{invsem_lemma_projdec}
      Every countable and inverse semigroup $S$ has a decreasing sequence of projections $\left\{e_n\right\}_{n \in \mathbb{N}}$ that is eventually below every other projection, that is, $e_n \geq e_{n + 1}$ and for every $f \in E(S)$ there is some $n_0 \in \mathbb{N}$ such that $f \geq e_{n_0}$.
    \end{lemma}
    \begin{proof}
      Since $S$ is countable we can enumerate the set of projections $E(S) = \left\{f_1, f_2, \dots\right\}$. The Lemma follows by letting $e_n := f_1 \dots f_n$.
    \end{proof}

    The localization property of the measure can be included in the reasoning leading to Theorem~\ref{invsem_th_ame0}, and this yields the following theorem.
    \begin{theorem}\label{invsem_th_ame01}
      Let $S$ be a countable and discrete inverse semigroup with identity $1 \in S$ and $\alpha \colon S\to \mathcal I (X)$ be a representation of $S$ on $X$. 
      Then the following conditions are equivalent:
      \begin{enumerate}
        \item \label{invsem_th_ame01_mea} $X$ is $S$-amenable.
        \item \label{invsem_th_ame01_par} $D_e$ is not $S$-paradoxical for any $e \in E\left(S\right)$.
        \item \label{invsem_th_ame01_fol} For every $\varepsilon > 0$ and finite $\mathcal{F} \subset S$ there is a finite non-empty $F \subset X$ such that $F \subset D_{s^*s}$ and $\left|s F \setminus F\right| < \varepsilon \left|F\right|$ for all $s \in \mathcal{F}$.
      \end{enumerate}
    \end{theorem}
    \begin{proof}
      Observe the equivalence between (\ref{invsem_th_ame01_mea}) and (\ref{invsem_th_ame01_par}) follows from Theorem~\ref{invsem_th_ame0} and Lemma~\ref{invsem_lemma_projdec}. Indeed, one can check that the proof of (\ref{invsem_th_ame0_mea}) $\Leftrightarrow$ (\ref{invsem_th_ame0_par}) in Theorem~\ref{invsem_th_ame0} works for any subset $A \subset X$, in particular if $A = D_e$. Thus, if $D_e$ is not paradoxical for any projection $e$ then there are measures $\mu_e$ on $X$ such that $\mu_e\left(D_e\right) = 1$. Now, by Lemma~\ref{invsem_lemma_projdec}, let $\left\{e_n\right\}_{n \in \mathbb{N}}$ be a decreasing sequence of projections that is eventually below every other projection. A measure in $X$ can be given by:
      \begin{equation}
        \mu\left(B\right) := \lim_{n \rightarrow \omega} \mu_{e_n}\left(B \cap D_{e_n}\right)\;,\quad B\subset X \;,\nonumber
      \end{equation}
      where $\omega$ is a free ultrafilter of $\mathbb{N}$. It is routine to show that $\mu$ is then a probability measure on $X$ satisfying the domain measurability and localization conditions mentioned above, i.e., $\mu\left(A\right) = \mu\left(sA\right)$ when $A \subset D_{s^*s}$ and $\mu\left(B\right) = \mu\left(B \cap D_{s^*s}\right)$ for every $s \in S$, $B \subset X$.

      In order to prove (\ref{invsem_th_ame01_fol}) $\Rightarrow$ (\ref{invsem_th_ame01_mea}) observe that the condition (\ref{invsem_th_ame01_fol}) ensures the existence of a domain F\o lner sequence $\left\{F_n\right\}_{n \in \mathbb{N}}$ such that for every $s \in S$ there is a number $N \in \mathbb{N}$ with 
      $F_n \subset D_{s^*s}$ for all $n \geq N$. Consider then a free ultrafilter $\omega$ on $\mathbb{N}$ and the measure
      \begin{equation}
        \mu\left(B\right) := \lim_{n \rightarrow \omega} \left|B \cap F_n\right|/\left|F_n\right|. \nonumber
      \end{equation}
      It follows from Theorem~\ref{invsem_th_ame0} that $\mu$ is a domain measure, which, in addition, satisfies that
      \begin{equation}
        \mu\left(D_{s^*s}\right) = \lim_{n \rightarrow \omega} \left|D_{s^*s} \cap F_n\right|/\left|F_n\right| = \lim_{n \rightarrow \omega} \left|F_n\right|/\left|F_n\right| = 1, \nonumber
      \end{equation}
      for all $s \in S$, i.e., the measure is localized.

      We will only sketch the proof (\ref{invsem_th_ame01_mea}) $\Rightarrow$ (\ref{invsem_th_ame01_fol}) since it is just a refinement of the same reasoning as in Theorem~\ref{invsem_th_ame0}. Let $\mu$ be an invariant measure on $X$. The corresponding mean $m \colon \ell^{\infty}(X) \rightarrow \mathbb{C}$ (see Proposition~\ref{invsem_lemma_mean}) 
      satisfies $m\left(P_{s^*s}\right) = 1$ for all $s \in S$. Then, any net $h_{\lambda}$ converging to $m$ in norm must also satisfy $\left|\left|h_{\lambda} \left(1 - P_{s^*s}\right)\right|\right|_1 \rightarrow 0$ for all $s \in S$. In particular, this must also be the case for the approximation $h$ appearing in Lemma~\ref{invsem_lemma_normapp}. To get the desired F\o lner set, we have to cut $h$ so that its whole support is within $D_{s^*s}$ for all $s \in \mathcal{F}$. For this, consider $\mathcal{F} = \left\{s_1, \dots, s_k\right\}$ and define the function
      \begin{equation}
        g := \dfrac{h \, P_{s_1^*s_1 \dots s_k^*s_k}}{\left|\left|h \, P_{s_1^*s_1 \dots s_k^*s_k}\right|\right|_1}\;. \nonumber
      \end{equation}
      The function $g$ has norm $1$, is positive and has finite support, which is contained in $D_{s^*s}$ for all $s \in \mathcal{F}$. Furthermore $\left|\left|h - g\right|\right|_1 \leq \varepsilon$. Thus, by substituting $h$ by $g$ in the proof of Theorem~\ref{invsem_th_ame0} and following the same construction, we obtain a F\o lner set $F$ within the support of $g$, that is, a F\o lner set within the requirements of the theorem.
    \end{proof}
    
    \begin{remark}
     Following results in \cite[Theorem~3.1]{GK17} one has that for inverse semigroups amenability is equivalent to the F\o lner condition and a local injectivity condition on the F\o lner sets. The preceding theorem is an improvement of Gray and Kambites' result applied to inverse semigroups. Amenability gives, in fact, that  
     F\o lner sets can be taken within the corresponding domains and, thus, the local injectivity condition is guaranteed by the localization property of the measures.
    \end{remark}

    Note that Theorem~\ref{invsem_th_ame01} is not constructive in the sense that if $X$ is not amenable, then one knows some $D_{e_0}$ is paradoxical, but Theorem~\ref{invsem_th_ame01} does not tell which element $e_0$ satisfies this condition. In the case $S$ has a minimal projection $e_0$ then one can improve the preceding Theorem. To do this we first remark the following simple and useful lemma.
    \begin{lemma}\label{invsem_lemma_mincomm}
      Let $S$ be a discrete and countable inverse semigroup with a minimal projection $e_0 \in E\left(S\right)$. Then $e_0$ commutes with every $s \in S$.
    \end{lemma}
    \begin{proof}
      Note that $e_0 = e_0 \, se_0s^* = s^*e_0s \, e_0$ by the minimality of $e_0$. Thus, we obtain
      \begin{equation}
        e_0 \, s = e_0 \, s e_0 s^* \, s = e_0 \, s \, e_0 = s \, s^*e_0s \, e_0 = s \, e_0, \nonumber
      \end{equation}
      where, for the first equality we have multiplied the identity $e_0=e_0se_0s^*$ from the right by $s$, and for the last equality we have multiplied the identity
      $e_0= s^*e_0se_0$ from the left by $s$. 
This proves the claim.
    \end{proof}
    \begin{proposition}\label{pro:invsem_cor_min}
      Let $S$ be a discrete, countable inverse semigroup with a minimal projection $e_0 \in E\left(S\right)$ and let $\alpha \colon S \rightarrow \mathcal{I}(X)$ be a representation. Then the following conditions are equivalent:
      \begin{enumerate}
        \item \label{pro:invsem_cor_min_ame} $X$ is $S$-amenable.
        \item \label{pro:invsem_cor_min_par} $D_{e_0}$ is not $S$-paradoxical.
        \item \label{pro:invsem_cor_min_fol} $D_{e_0}$ is $S$-domain F\o lner.
      \end{enumerate}
    \end{proposition}
    \begin{proof}
      The implication (\ref{pro:invsem_cor_min_ame}) $\Rightarrow$ (\ref{pro:invsem_cor_min_par}) is a particular case of Theorem~\ref{invsem_th_ame01}. For (\ref{pro:invsem_cor_min_par}) $\Rightarrow$ (\ref{pro:invsem_cor_min_fol}) note that it follows from Lemma~\ref{invsem_lemma_mincomm} that $D_{e_0} \subset X$ is an invariant subset for the action. Indeed, for any $s \in S$
      \begin{equation}
        s \left(D_{e_0} \cap D_{s^*s}\right) = s D_{e_0} \subset D_{e_0}. \nonumber
      \end{equation}
      The implication (\ref{pro:invsem_cor_min_par}) $\Rightarrow$ (\ref{pro:invsem_cor_min_fol}) therefore follows from Theorem~\ref{invsem_th_ame0} by considering, if necessary, the induced action of $S$ on $D_{e_0}$. Finally, (\ref{pro:invsem_cor_min_fol}) $\Rightarrow$ (\ref{pro:invsem_cor_min_ame}) can be proven in a similar fashion to that of (\ref{invsem_th_ame01_fol}) $\Rightarrow$ (\ref{invsem_th_ame01_mea}) in Theorem~\ref{invsem_th_ame01}. 
    \end{proof}

    As an example of an inverse semigroup $S$ with a minimal projection we consider the case where $S$ satisfies the F\o lner condition but {\em not} the {\em proper} F\o lner condition (see Definition~\ref{def:sem}).
    \begin{proposition}\label{invsem_fol_notpropfl}
      Let $S$ be a countable and discrete inverse semigroup. Suppose $S$ satisfies the F\o lner condition but not the proper F\o lner condition. Then $S$ has a minimal projection.
    \end{proposition}
    \begin{proof}
      Following Theorem~\ref{sem_fvspf} there is an element $a \in S$ such that $\left|S a\right| < \infty$. Suppose $Sa = \left\{s_1a, \dots, s_ka\right\}$. Then we claim $e := s_1^*s_1 \dots s_k^*s_k aa^* \in S$ is a minimal projection. Indeed, for any other projection $f \in E\left(S\right)$ there is an $i$ such that $s_ia = fa$. In this case we have
      \begin{equation}
        a^*fa = \left(fa\right)^* fa = \left(s_ia\right)^*s_ia = a^*s_i^*s_ia. \nonumber
      \end{equation}
      Thus, multiplying by $a$ from the left, $aa^*fa = fa = aa^*s_i^*s_ia = s_i^*s_ia$. Therefore 
      \[
        f \geq faa^* = s_i^*s_iaa^* \geq e\;,
      \]
      proving that $e$ is indeed minimal.
    \end{proof}

    Note that, in order to produce an example of an inverse semigroup $S$ that is F\o lner but not proper F\o lner, as in the hypothesis of Proposition~\ref{invsem_fol_notpropfl}, $S$ must have a minimal projection $e_0 \in E(S)$. Moreover, by Proposition~\ref{pro:invsem_cor_min}, the domain $D_{e_0}$ must be domain-F\o lner.
    It can thus be shown that $S := \mathbb{F}_2 \sqcup \left\{0\right\}$ satisfies the F\o lner condition but not the proper one.

  %----------------------------------------------------------------------------------------
  % SECTION: INVERSE SEMIGROUPS, C*-ALGEBRAS AND TRACES
  %----------------------------------------------------------------------------------------
  \section{Inverse semigroups, C$^*$-algebras and traces}\label{sec_invsem_ctr}
  In this final part of the article we connect the analysis of the previous section with properties of a C*-algebra $\mathcal{R}_X$ generalizing the uniform Roe algebra $\mathcal{R}_G$ of a group $G$ presented in Section~\ref{sec_pre}. In particular, we will show that domain measurability completely characterizes the existence of traces in $\mathcal{R}_X$.

  Let $S$ be a discrete inverse semigroup with identity and consider a representation $\alpha \colon S\to\mathcal{I}(X)$ on a set $X$. As before, we will denote $\alpha_s(x)$ simply by $sx$ for any $s\in S$, $x\in D_{s^*s}$. To construct the C$^*$-algebra $\mathcal{R}_X$ consider first the representation $V \colon S \rightarrow \mathcal{B}\left(\ell^2 (X)\right)$ given by
  \begin{equation}
    V_s \delta_x := \left\{ \begin{array}{lcc}
                      \delta_{sx} & \text{if} & x\in D_{s^*s} \\
                      0 & \text{if} & x\notin D_{s^*s}\;,
                    \end{array} \right. \nonumber
  \end{equation}
  where $\left\{\delta_x\right\}_{x \in X} \subset \ell^2(X)$ is the canonical orthonormal basis. $V$ is a $*$-representation of $S$ in terms of partial isometries of $\ell^2 (X)$. Define the unital *-subalgebra $\mathcal{R}_{X, alg}$ in $\mathcal{B}(\ell^2(X))$
  generated by $\left\{V_s \mid s \in S\right\}$ and $\ell^{\infty}(X)$. The C*-algebra $\mathcal{R}_X$ is defined as the norm closure of $\mathcal{R}_{X, alg}$, i.e., 
  \begin{equation}
    \mathcal{R}_X := \overline{\mathcal{R}_{X, alg}} = \text{C}^*\Big(\left\{V_s\mid s\in S\right\}\cup \ell^\infty(X)\Big)\subset\mathcal{B}\left(\ell_2(X)\right)\;. \nonumber
  \end{equation}

  Note that conjugation by $V_s$ implements the action of $s \in S$ on subsets $A \subset X$. It is straightforward to show the following intertwining equation for the generators of $\mathcal R_X$: 
  \[
    P_A V_s = V_s P_{s^*(A\cap D_{ss^*})}\;,\quad s\in S\;,\; A\subset X\;,
  \]
  where, as before, $P_A \in \mathcal{B}\left(\ell^2(X)\right)$ is the orthogonal projection onto the closure of $\text{span}\left\{\delta_a \mid a \in A\right\}$. Note also that for any $s \in S$, $f\in\ell^\infty(X)$ (which we interpret as multiplication operators on $\ell^2(X)$) we have the following commutation relations between the generators of $\mathcal{R}_X$
  \begin{equation}
    V_s \, f = \big(s f\big)\, V_s \quad\mathrm{or, equivalently,} \quad f\, V_s  = V_s\, \big(s^{*} f\big)\;. \nonumber
  \end{equation}
  From this commutation relations, it follows that the algebra $\mathcal{R}_X$ is actually the closure in operator norm of the linear span of operators of the form $V_s P_A$, where $s \in S$, $A \subset D_{s^*s}$. This fact will be used throughout the section.

    %----------------------------------------------------------------------------------------
    % SUBSECTION: DOMAIN-MEASURES AS TRACES
    %----------------------------------------------------------------------------------------
    \subsection{Domain measures as amenable traces}
    Before proving the next theorem we first need to introduce some notation and some preparing lemmas. The first result defines a canonical conditional expectation from $\mathcal{B}\left(\ell^2(X)\right)$ onto $\ell^{\infty}(X)$. The proof is virtually the same as in the case of groups (see, e.g., \cite{BO08}).
    \begin{lemma}\label{invsem_lemma_exp}
      The linear map $\mathcal{E} \colon \mathcal{B}\left(\ell^2(X)\right) \rightarrow \ell^{\infty}(X)$ given by
      \begin{equation}
        \mathcal{E}\left(T\right) = \sum_{x \in X} P_{\left\{x\right\}} T P_{\left\{x\right\}}, \nonumber 
      \end{equation}
      is a conditional expectation, where the sum is taken in the strong operator topology.
    \end{lemma}

    To analyze the dynamics on F\o lner sets of inverse semigroups it is convenient to introduce the following
    equivalence relation in $X$. Let $\alpha \colon S\to \mathcal I (X)$ be a representation of $S$ on $X$. Given a pair $u, v \in X$, we write $u \approx v$ if  there is some $s \in S$ such that $u \in D_{s^*s}$ and $su = v$. The relation $\approx$ is an equivalence relation: in fact, since $S$ is unital, $u \approx u$, and if $u \approx v$ then $v \approx u$ by considering $v = su \in D_{ss^*}$. For transitivity, if $u \approx v \approx w$ then $u \in D_{s^*t^*ts}$, where $s, t$ witness $u \approx v$ and $v \approx w$ respectively. We will see in Lemma~\ref{invsem_lemma_dimw} that if a set $F \subset X$ has only one $\approx$-class, then the corner $P_F \mathcal{R}_X P_F$ has dimension $\left|F\right|^2$ as a vector space.

    The next lemma guarantees the existence of transitive domain F\o lner sets, i.e., of domain F\o lner sets $F$ such that $F/\approx$ is a singleton.
    \begin{lemma}\label{invsem_lemma_fequiv}
      Let $\alpha \colon S\to \mathcal I (X)$ be a representation of $S$ on $X$. If $A \subset X$ is domain F\o lner then for every $\varepsilon > 0$ and finite $\mathcal{F} \subset S$, there is a 
      $(\varepsilon , \mathcal F)$-domain F\o lner $F_0 \subset A$ with exactly one $\approx$-equivalence class
    \end{lemma}
    \begin{proof}
      Since $A$ is domain F\o lner, for any $\varepsilon > 0$ and finite $\mathcal{F} \subset S$ there is a finite $F \subset A$ such that
      \begin{equation}
        \left|s\left(F \cap D_{s^*s}\right) \setminus F\right| < \frac{\varepsilon}{\left|\mathcal{F}\right|}\, \left|F\right| \, , \quad \text{for all} \; s \in \mathcal{F}. \nonumber
      \end{equation}
      Decomposing $F$ into its $\approx$-classes we get $F = F_1 \sqcup \dots \sqcup F_L$, where $u \approx v$ if and only if $u, v \in F_i$ for some $i$. To prove the claim it is enough to prove that some $F_j$ must be $\left(\varepsilon, \mathcal{F}\right)$-domain F\o lner. Indeed, suppose for all $j = 1, \dots, L$ there is an $s_j \in \mathcal{F}$ such that
      \begin{equation}
        \left|s_j\left(F_j \cap D_{s_j^*s_j}\right) \setminus F_j\right| \geq \varepsilon \left|F_j\right|. \nonumber
      \end{equation}
      Observe that the choice of $s_j$ is not unique, but we can consider a particular fixed choice. Arrange then the indices according to the following: 
      for $s \in \mathcal{F}$, consider $\Lambda_s := \left\{j \in \{ 1, \dots, L \} \mid s_j = s\right\}$. Note that some $\Lambda_s$ might be empty. Define $F_s := \sqcup_{i \in \Lambda_s} F_i$ and observe
      \begin{equation}
        \left|s \left(F_s \cap D_{s^*s}\right) \setminus F_s\right| = \sum_{j \in \Lambda_s} \left|s_j \left(F_j \cap D_{s_j^*s_j}\right) \setminus F_j\right| \geq \varepsilon \sum_{j \in \Lambda_s} \left|F_j\right| = \varepsilon \left|F_s\right|. \nonumber
      \end{equation}
      Taking the sum over all $s \in \mathcal{F}$  we get
      \begin{align}
        \varepsilon \left|F\right| > \sum_{s \in \mathcal{F}} \left|s\left(F \cap D_{s^*s}\right) \setminus F\right| & = \sum_{s, t \in \mathcal{F}} \left|s\left(F_t \cap D_{s^*s}\right) \setminus F_t\right| \nonumber \\
        & \geq \sum_{s \in \mathcal{F}} \left|s \left(F_s \cap D_{s^*s}\right) \setminus F_s\right| \geq \varepsilon \sum_{s \in \mathcal{F}} \left|F_s\right| = \varepsilon \left|F\right|. \nonumber
      \end{align}
      This is a contradiction and, thus, some $F_{j_0}$ must witness the domain F\o lner condition.
    \end{proof}

    The next lemma computes the dimension of a certain corner of the algebra $\mathcal{R}_X$.
    \begin{lemma}\label{invsem_lemma_dimw}
      Let $\alpha \colon S\to \mathcal I (X)$ be a representation of $S$ on $X$. Let $F_1, F_2 \subset X$ be finite sets such that $F_1 \cup F_2$ has only one $\approx$-class. Then $W := P_{F_2} \mathcal{R}_X P_{F_1}$ has linear dimension $\left|F_1\right| \left|F_2\right|$.
    \end{lemma}
    \begin{proof}
      To prove the claim it suffices to show that for every $u_i \in F_i, i = 1, 2$ the matrix unit
      \begin{equation}
        M_{u_2, u_1} \delta_x = \left\{
                                        \begin{array}{lcc}
                                          \delta_{u_2} & \text{if} & x = u_1 \nonumber \\
                                          0 &  & \text{otherwise,}
                                        \end{array}
                                      \right. \nonumber
      \end{equation}
      is contained in $W$. Since $u_1 \approx u_2$ there must be an element $s \in S$ such that $u_1 \in D_{s^*s}$ and $su_1 = u_2$. It is straightforward to prove that in this case
      \begin{equation}
        M_{u_2, u_1} = P_{F_2} P_{u_2} V_s P_{u_1} P_{F_1} \in P_{F_2} \mathcal{R}_X P_{F_1} = W\;, \nonumber
      \end{equation}
      hence dim$W=\left|F_1\right| \left|F_2\right|$.
    \end{proof}

    We are now in a position to show the main theorem of the section, which characterizes the domain measurability of the action in terms of amenable 
    traces of the algebra $\mathcal{R}_X$.
    \begin{theorem}\label{invsem_th_ame1}
      Let $S$ be a countable and discrete inverse semigroup with identity $1 \in S$, and let $\alpha \colon S\to \mathcal I (X)$ be a representation of $S$ on $X$. Then the following are equivalent:
      \begin{enumerate}
        \item \label{invsem_th1_dommeas} $X$ is $S$-domain measurable.
        \item \label{invsem_th1_dompar}  $X$ is not $S$-paradoxical.
        \item \label{invsem_th1_fol} $X$ is $S$-domain F\o lner.
        \item \label{invsem_th1_algame} $\mathcal{R}_{X, alg}$ is algebraically amenable.
        \item \label{invsem_th1_roeamtr} $\mathcal{R}_X$ has an amenable trace (and hence is a F\o lner C$^*$-algebra).
        \item \label{invsem_th1_roeinf}  $\mathcal{R}_X$ is not properly infinite.
        \item \label{invsem_th1_roekgr}  $\left[0\right] \neq \left[1\right]$ in the K$_0$-group of $\mathcal{R}_X$.
      \end{enumerate}
    \end{theorem}
    \begin{proof}
      The equivalences (\ref{invsem_th1_dommeas}) $\Leftrightarrow$ (\ref{invsem_th1_dompar}) $\Leftrightarrow$ (\ref{invsem_th1_fol}) follow from Theorem~\ref{invsem_th_ame0}.

      (\ref{invsem_th1_dommeas}) $\Rightarrow$ (\ref{invsem_th1_roeamtr}). Consider the conditional expectation $\mathcal{E} \colon \mathcal{B}\left(\ell^2(X)\right) \rightarrow \ell^{\infty}(X)$ introduced in 
      Lemma~\ref{invsem_lemma_exp}. Since $X$ is $S$-domain measurable, by Proposition~\ref{invsem_lemma_mean}, there is a mean $m \colon \ell^{\infty}\left( X\right)$ $\rightarrow \mathbb{C}$ 
      satisfying $m\left(sf\right) = m\left(s^*sf\right)$. We claim that then $\phi\left(T\right) := m\left(\mathcal{E}\left(T\right)\right)$ is a hypertrace on $\mathcal{R}_X$. Indeed, observe that linearity, positivity and normalization follow from those of $m$ and $\mathcal{E}$. Hence we only have to prove the hypertrace property for the generators of $\mathcal{R}_X$. Note that since $\mathcal{E}$ is a conditional expectation we have $\phi\left(f T \right)=\phi\left(T f\right)$ for any $f\in \ell^{\infty}(X)$, $T \in \mathcal{B}\left(\ell^2(X)\right)$. To show the same relation for the generator $V_s$ note first that for any $s, t \in S$ the following relation holds:
      \begin{align}
        \mathcal{E}\left(V_s T\right) \left(x\right) & = \left\{  \begin{array}{lcc}
                                                          T_{s^*x, x} & \text{if} & x\in D_{ss^*} \\
                                                          0 & \text{if} & x\notin D_{ss^*}
                                                        \end{array} \right. \nonumber \\
        \mathcal{E}\left(T V_s\right) \left(y\right) & = \left\{  \begin{array}{lcc}
                                                          T_{y, sy} & \text{if} & y\in D_{s^*s} \\
                                                          0 & \text{if} & y\notin D_{s^*s}
                                                        \end{array} \right. .\nonumber
      \end{align}
      It follows from the action introduced in Eq.~(\ref{eq:action-S}) that 
      $s \; \mathcal{E} (TV_s)= \mathcal{E}(V_sT)$ and  $\mathcal{E} (TV_s) = \mathcal{E} (TV_s) P_{s^*s}$ and thus
      \begin{equation}
        \phi\left(V_sT \right)  = m\left(\mathcal{E}\left(V_sT \right)\right) = m\left(s\cdot \mathcal{E}\left(TV_s \right)\right) =  m\left(\mathcal{E}\left(TV_s \right) P_{s^*s}\right) =   m\left(\mathcal{E}\left(TV_s \right)\right)  =  \phi\left(T V_s\right), \nonumber
      \end{equation}
      where we used the invariance of the mean in the third equality. 

      (\ref{invsem_th1_roeamtr}) $\Rightarrow$ (\ref{invsem_th1_roeinf}). Suppose that $\mathcal{R}_X$ is properly infinite \textit{and} has a hypertrace $\phi$. Then we obtain a contradiction from 
      \begin{equation}
        1 = \phi\left(1\right) \geq \phi\left(W_1W_1^*\right) + \phi\left(W_2W_2^*\right) = \phi\left(W_1^*W_1\right) + \phi\left(W_2^*W_2\right) = \phi\left(1\right) + \phi\left(1\right) = 2, \nonumber
      \end{equation}
      where $W_1, W_2$ are the isometries witnessing the proper infiniteness of $1 \in \mathcal R_X$.

      (\ref{invsem_th1_roeinf}) $\Rightarrow$ (\ref{invsem_th1_dompar}). Suppose that $s_i, t_j, A_i, B_j$, $i = 1,\dots,n$, $j=1,\dots,m$, implement the $S$-paradoxicality of $X$, that is, $A_i\subseteq D_{s_i^*s_i}$, $B_j\subseteq D_{t_j^*t_j}$, and
      \begin{align}
        X & = s_1 A_1 \, \sqcup \, \dots \, \sqcup \, s_n A_n = t_1 B_1 \, \sqcup \, \dots \, \sqcup \, t_m B_m \nonumber \\
        & \supset  A_1 \, \sqcup \, \dots \, \sqcup \,  A_n \, \sqcup \, B_1 \, \sqcup \, \dots \, \sqcup \,  B_m. \nonumber
      \end{align}
      Consider now the operators
      \begin{equation}
        W_1 := V_{s_1^*} P_{s_1 A_1} + \dots + V_{s_n^*} P_{s_n A_n} \quad \text{and} \quad W_2 := V_{t_1^*} P_{t_1 B_1} + \dots + V_{t_m^*} P_{t_m B_m}. \nonumber
      \end{equation}
      These are both partial isometries, since $V_{s_i^*} P_{s_i A_i}$ and $V_{t_j^*} P_{t_j B_j}$ are partial isometries with pairwise orthogonal domain and range projections. Furthermore, $W_1^*W_1$ is the projection onto the union of the domains of $V_{s_i^*} P_{s_i A_i}$, which is 
      the whole space $\ell^2(X)$. The same argument proves that $W_2^*W_2 = 1$. Therefore, to prove the claim we just have to show that $W_1W_1^*$ and $W_2W_2^*$ are orthogonal projections. But these correspond to projections onto $\sqcup_{i = 1}^n  A_i$ and $\sqcup_{j = 1}^m B_j$, 
      respectively, which are disjoint sets in $X$.

      (\ref{invsem_th1_fol}) $\Rightarrow$ (\ref{invsem_th1_algame}). By Lemma~\ref{invsem_lemma_fequiv}, given $\varepsilon > 0$ and finite $\mathcal{F} \subset S$, there is a finite non-empty $F \subset X$ with exactly one $\approx$-class and such that
      \begin{equation}
        \left|s\left(F \cap D_{s^*s}\right) \setminus F\right| < \varepsilon \left|F\right|, \quad \text{for all} \; s \in \mathcal{F}. \nonumber
      \end{equation}
      Consider the space $W := P_F \mathcal{R}_X P_F = P_F \mathcal{R}_{X, alg} P_F$ and observe
      \begin{align}
        V_s W & = \left\{ V_s \, P_F \, T \, P_F \; \mid \; T \in \mathcal{R}_{X, alg} \right\} \nonumber \\
        & = \left\{ P_{s\left(F \cap D_{s^*s}\right)} \, V_s \, T \, P_F \; \mid \; T \in \mathcal{R}_{X, alg} \right\} \nonumber \\
        & \subset \left\{ P_{s\left(F \cap D_{s^*s} \setminus F\right)} \, V_s \, T \, P_F \; \mid \; T \in \mathcal{R}_{X, alg} \right\}  + \left\{ P_F \, P_{s\left(F \cap D_{s^*s}\right)} \, V_s \, T \, P_F \; \mid \; T \in \mathcal{R}_{X, alg} \right\}. \nonumber
      \end{align}
      Therefore
      \begin{equation}
        V_s W + W \subset P_{s\left(F \cap D_{s^*s}\right) \setminus F} \Big( \mathcal{R}_{X, alg} \Big) P_F + P_F \left( \mathcal{R}_{X, alg} \right) P_F\;, \nonumber
      \end{equation}
      and, by Lemma~\ref{invsem_lemma_dimw},
      \begin{equation}
        \text{dim}\left(W + V_s W\right) \leq \left|F\right|^2 + \left|F\right| \, \left|s\left(F \cap D_{s^*s}\right) \setminus F\right| \leq \left(1 + \varepsilon\right) \text{dim} \left(W\right), \nonumber
      \end{equation}
      which proves the algebraic amenability of $\mathcal{R}_{X, alg}$.

      (\ref{invsem_th1_algame}) $\Rightarrow$ (\ref{invsem_th1_roeamtr}). This follows from one of the main results of \cite{ALLW18R}, 
      which states that if a pre-C$^*$-algebra is algebraically amenable then its closure has an amenable trace (see \cite[Theorem~3.17]{ALLW18R}) and 
      hence is a F\o lner C*-algebra (cf., Definition~\ref{def:F-alg}~(\ref{item:1})).

      (\ref{invsem_th1_roeamtr}) $\Rightarrow$ (\ref{invsem_th1_roekgr}). Any trace $\phi$ on $\mathcal{R}_X$, by the universal property of the K$_0$ group, induces a group homomorphism $\phi_0 \colon K_0\left(\mathcal{R}_X\right) \rightarrow \mathbb{R}$ such that $\phi_0\left(\left[P\right]\right) = \phi\left(P\right)$ for any projection $P \in \mathcal{R}_X$. In this case $\phi_0\left(\left[1\right]\right) = \phi\left(1\right) = 1$ while $0 = \phi\left(0\right) = \phi_0\left(\left[0\right]\right)$. In particular, $\left[1\right] \neq \left[0\right]$ in the K$_0$ group of $\mathcal{R}_X$.

      (\ref{invsem_th1_roekgr}) $\Rightarrow$ (\ref{invsem_th1_dompar}). If $X$ is $S$-paradoxical, then it follows from Lemma~\ref{invsem_lemma_type}(3) and the same argument as in the implication (\ref{invsem_th1_roeinf}) $\Rightarrow$ (\ref{invsem_th1_dompar}) that $[1] = [1]+ [1]$ in $K_0(\mathcal R _X)$. Therefore $[1]=[0]$ in $K_0(\mathcal R _X)$.
    \end{proof}

    Theorem~\ref{invsem_th_ame1} can be generalized, along the lines of \cite{RS12}, to hold for any set $A \subset X$.
    \begin{corollary}\label{invsem_cor_ame_x}
      Let $S$ be a countable and discrete inverse semigroup with identity $1 \in S$, and let $\alpha \colon S\to \mathcal I (X)$ be  a representation of $S$ on $X$. Let $A \subset X$, then the following conditions are equivalent:
      \begin{enumerate}
        \item \label{invsem_cor_ame_meas} $A$ is $S$-domain measurable.
        \item \label{invsem_cor_ame_par} $A$ is not $S$-paradoxical.
        \item \label{invsem_cor_ame_weight} There is a tracial weight $\psi \colon \mathcal{R}_X^+ \rightarrow \left[0, \infty\right]$ such that $\psi\left(P_A\right) = 1$.
        \item \label{invsem_cor_ame_propinf} $P_A \in \mathcal{R}_X$ is not a properly infinite projection.
      \end{enumerate}
    \end{corollary}
    \begin{proof}
      Most of the proof is similar as in the reasoning leading to Theorem~\ref{invsem_th_ame1}. The only difference regards condition (\ref{invsem_cor_ame_weight}), and the replacement of the trace with a weight. To prove that (\ref{invsem_cor_ame_weight}) $\Rightarrow$ (\ref{invsem_cor_ame_meas}), consider the measure $\mu\left(B\right) = \psi\left(P_B\right)$. This $\mu$ will then be a measure on $X$ such that $\mu\left(A\right) = 1$. Furthermore, invariance follows from $\psi$ being tracial:
      \begin{equation}\label{eq_trace_meas}
        \mu\left(B\right) = \psi\left(V_{s^*s} P_B\right) = \psi\left(V_s P_B V_{s^*}\right) = \psi\left(P_{sB}\right) = \mu\left(\alpha_s\left(B\right)\right)
      \end{equation}
      for every $B \subset D_{s^*s}$.

      To prove (\ref{invsem_cor_ame_meas}) $\Rightarrow$ (\ref{invsem_cor_ame_weight}) we adapt the ideas in \cite[Proposition~5.5]{RS12}. Given a domain measure $\mu$ normalized at $A$, denote by $\mathcal{P}_{fin}(X)$ the (upwards directed) set of $K \subset X$ with finite measure, i.e., $\mu\left(K\right) < \infty$. Given $K \in \mathcal{P}_{fin}(X)$ consider the finite measure $\mu_K\left(B\right) = \mu\left(K \cap B\right)$ and extend it, as in Proposition~\ref{invsem_lemma_mean}, to a functional $m_K$. Given a non-negative $f \in \ell^{\infty}(X)$ define
      \begin{equation}
        m\left(f\right) := \sup_{K \in \mathcal{P}_{fin}(X)} m_K\left(f\right). \nonumber
      \end{equation}
      Then $m$ is $\mathbb{R}_+$-linear, lower-semicontinuous, normalized at $P_A$, and satisfies that $m\left(sf\right) = m\left(s^*sf\right)$, for every $s \in S, f \in \ell^{\infty}(X)$. Finally, the weight given by $\psi := m \circ \mathcal{E} \colon \mathcal{R}_X^+ \rightarrow \left[0, \infty\right]$ is a tracial weight.
    \end{proof}
    
    In our last result of the section we point out that one can translate the F\o lner sequences of $X$ onto F\o lner sequences of projections in $\mathcal{R}_X$ (cf., Definition~\ref{def:F-alg}(\ref{F-alg-2})). To prove it we first recall the following known result, which states that F\o lner sequences are preserved under C*-closure (see, e.g., \cite{B97}).
    \begin{lemma}\label{lem_fsclosure}
      Let $\mathcal{T} \subset \mathcal{B}\left(\mathcal{H}\right)$ be a set of operators on a separable Hilbert space $\mathcal{H}$. Suppose $\mathcal{T}$ has a F\o lner sequence $\left\{P_n\right\}_{n = 1}^{\infty}$. Then $\left\{P_n\right\}_{n = 1}^{\infty}$ is a F\o lner sequence for the C$^*$-algebra generated by $\mathcal{T}$.
    \end{lemma}
    \begin{proposition}\label{invsem_th_folnsqin}
      Let $S$ be a countable and discrete inverse semigroup and $\alpha \colon S \rightarrow \mathcal{I}(X)$ be a representation. If $X$ is domain measurable, then there is a sequence of projections 
      $\left\{P_n\right\}_{n = 1}^{\infty} \subset \mathcal{R}_X$ which is a F\o lner sequence of projections for each operator $T \in \mathcal{R}_X$.
    \end{proposition}
    \begin{proof}
      Choose a $S$-domain F\o lner sequence $\left\{F_n\right\}_{n = 1}^{\infty}$ of $X$, that exists since $X$ is domain measurable (see Theorem~\ref{invsem_th_ame1}). Consider the orthogonal projection $P_n$ onto $\text{span}\left\{\delta_f \mid f \in F_n\right\} \subset \ell^2(X)$. Clearly, $P_n$ lies within $\mathcal{R}_X$, so, by Lemma~\ref{lem_fsclosure}, it is enough to show that it is a F\o lner sequence for all generating elements $V_sP_A$, $s\in S$, $A\subset X$. For this we compute
      \begin{align}
        \left(V_s P_A P_n\right) \left(\delta_x\right) & = \left\{  \begin{array}{lcc}
                                                          \delta_{sx} \;,& \text{if} & x \in D_{s^*s} \cap F_n \cap A \\
                                                          0 \;, & & \text{otherwise}
                                                        \end{array} \right. \nonumber \\
        \left(P_n V_s P_A\right) \left(\delta_x\right) & = \left\{  \begin{array}{lcc}
                                                          \delta_{sx}\;, & \text{if} & x \in A \cap s^*\left(F_n  \cap D_{ss^*}\right) \\
                                                          0\;, & & \text{otherwise}.
                                                        \end{array} \right. \nonumber
      \end{align}
      Thus we have the following estimates in the Hilbert-Schmidt norm:
      \begin{align}
        \left|\left|V_s P_A P_n - P_n V_s P_A\right|\right|_2^2 & \leq \left|\left\{x \in F_n \cap D_{s^*s} \mid sx \not\in F_n\right\}\right|
                                                                     + \left|s^*\left(F_n \cap D_{ss^*}\right) \setminus F_n\right| \nonumber \\
        & = \left|s\left(F_n \cap D_{s^*s}\right) \setminus F_n\right| + \left|s^*\left(F_n \cap D_{ss^*}\right) \setminus F_n\right|. \nonumber
      \end{align}
      Noting that $\left|\left|P_n\right|\right|_2^2 = \left|F_n\right|$ the result follows from normalizing by $\left|F_n\right|$ and taking limits on both sides
      of the inequality.
    \end{proof}

    %----------------------------------------------------------------------------------------
    % SUBSECTION: TRACES VS AMENABLE TRACES
    %----------------------------------------------------------------------------------------
    \subsection{Traces and amenable traces}
    The aim of this section is to prove that either all traces in $\mathcal{R}_X$ are amenable or this algebra has no traces at all. Similar results hold for nuclear C*-algebras (see \cite[Proposition~6.3.4]{BO08}) and for
    uniform Roe algebras over metric spaces (see \cite[Corollary~4.15]{ALLW18R}). We begin by proving that traces of $\mathcal{R}_X$ factor through $\ell^{\infty}(X)$ via the conditional expectation $\mathcal{E}$ (see Lemma~\ref{invsem_lemma_exp}).
    \begin{lemma}\label{invsem_lemma_trcond}
      Let $\alpha \colon S \rightarrow \mathcal{I}(X)$ be a representation and let $\phi \colon \mathcal{R}_X \rightarrow \mathbb{C}$ be a trace. Then $\phi\left(T\right) = \phi\left(\mathcal{E}\left(T\right)\right)$ for every $T \in \mathcal{R}_X$, where $\mathcal{E}$ denotes the canonical conditional expectation onto $\ell^{\infty}(X)$.
    \end{lemma}
    \begin{proof}
      Since the closure of the linear span of the elements of the form $V_s P_A$, $s \in S$, $A \subset D_{s^*s}$, is dense in $\mathcal{R}_X$ it is enough to show the claim for these elements.

      First suppose that $A$ has no fixed points under $s$, i.e., $\left\{a \in A \mid sa = a\right\} = \emptyset$. Consider the graph whose vertices are the elements of $A$ and such that two vertices $a, b$ are joined by an edge if and only if $b = sa$. Since the action of $s$ is injective on $A$ it is clear that every vertex has at most degree $2$ and no loops. Therefore it can be colored by $3$ colors and, thus, there is a partition $A=B_1 \sqcup B_2 \sqcup B_3$ such that if $a \in B_i$ then $sa \not\in B_i$. This allows us to decompose $V_s P_A$ as
      \begin{equation}
        V_s P_A = V_s P_{B_1} + V_s P_{B_2} + V_s P_{B_3}. \nonumber
      \end{equation}
      Taking traces on each side of the equality gives
      \begin{equation}
        \phi\left(V_s P_A\right) = \phi\left(P_{B_1} V_s P_{B_1}\right) + \phi\left(P_{B_2} V_s P_{B_2}\right) + \phi\left(P_{B_3} V_s P_{B_3}\right) = 0\;. \nonumber
      \end{equation}
      But in this case we also have $\mathcal{E}\left(V_s P_A\right) = 0$ and the equality $\phi=m\circ\mathcal{E}$ follows.

      Second, for arbitrary $s \in S$ and $A \subset D_{s^*s}$ one can decompose $A = B \sqcup C$, where $B := \left\{a \in A \mid sa = a\right\}$ is the set of fixed points and $C := \left\{a \in A \mid sa \neq a \right\}$. In this case
      \begin{equation}
        V_s P_A = V_s P_B + V_s P_C. \nonumber
      \end{equation}
      By the above paragraph it follows that $\phi\left(V_sP_C\right) = 0$, while $V_s P_B$ is a projection in $\ell^{\infty}\left(S\right)$. Thus
      \begin{equation}
        \phi\left(V_s P_A\right) = \phi\left(V_s P_B\right) + \phi\left(V_s P_C\right) = \phi\left(\mathcal{E}\left(V_s P_B\right)\right) = \phi\left(\mathcal{E}\left(V_s P_A\right)\right), \nonumber
      \end{equation}
      which concludes the proof.
    \end{proof}
    The following theorem is a consequence of Theorem~\ref{invsem_th_ame1} and the latter lemma.
    \begin{theorem}\label{invsem_thtr}
      Let $S$ be a countable and discrete inverse semigroup with identity $1 \in S$, and let $\alpha \colon S\to \mathcal I (X)$ be  a representation. Consider a positive linear functional $\phi$ on $\mathcal{R}_X$. Then the following conditions are equivalent:
      \begin{enumerate}
        \item \label{invsem_thtr_roeamtr} $\phi$ is an amenable trace on $\mathcal{R}_X$.
        \item \label{invsem_thtr_roetr}   $\phi$ is a trace on $\mathcal{R}_X$.
        \item \label{invsem_thtr_form}    $\phi = \phi |_{\ell^{\infty}(X)} \circ \mathcal{E}$ and the measure $\mu\left(A\right) := \phi\left(P_A\right)$ satisfies domain measurability, i.e., $\mu\left(A\right) = \mu\left(sA\right)$ for all $s \in S$, $A \subset D_{s^*s}$.
      \end{enumerate}
    \end{theorem}
    \begin{proof}
      The fact that (\ref{invsem_thtr_roeamtr}) $\Rightarrow$ (\ref{invsem_thtr_roetr}) is obvious. For (\ref{invsem_thtr_roetr}) $\Rightarrow$ (\ref{invsem_thtr_form}) note that by Lemma~\ref{invsem_lemma_trcond} it remains only to prove that $\mu\left(A\right) = \mu\left(sA\right)$ for every $s \in S$, $A \subset D_{s^*s}$. This follows from $\phi$ being a trace and Eq.~(\ref{eq_trace_meas}). Finally, (\ref{invsem_thtr_form}) $\Rightarrow$ (\ref{invsem_thtr_roeamtr}) is proved as the implication (\ref{invsem_th1_dommeas}) $\Rightarrow$ (\ref{invsem_th1_roeamtr}) in Theorem~\ref{invsem_th_ame1}.
    \end{proof}

    We summarize next some important consequences of the previous theorems.
    \begin{corollary}
      Let $S$ be a countable and discrete inverse semigroup with identity $1 \in S$, and let $\alpha \colon S\to \mathcal I (X)$ be a representation. Then:
      \begin{enumerate}
        \item $X$ is $S$-domain measurable if and only if there is a trace on $\mathcal{R}_X$, in which case every trace on $\mathcal{R}_X$ is amenable.
        \item There is a canonical bijection between the space of measures on $X$ such that $\mu\left(A\right) = \mu\left(sA\right)$ when $s \in S$, $A \subset D_{s^*s}$ and the space of traces of $\mathcal{R}_X$.
      \end{enumerate}
    \end{corollary}

    %----------------------------------------------------------------------------------------
    % SUBSECTION: TRACES IN AMENABLE INVERSE SEMIGROUPS
    %----------------------------------------------------------------------------------------
    \subsection{Traces in amenable inverse semigroups}\label{sec_invsem_ame}
    The goal of this last section is to state the analogue of Theorem~\ref{invsem_th_ame1} but considering amenable semigroups instead of the weaker notion  of domain measurable ones. That is, we will give additional C*-algebraic characterizations of amenable inverse semigroups than those in Theorem~\ref{invsem_th_ame01}. The proof of the following result is straightforward (see the proof of Theorem~\ref{invsem_th_ame1}).
    \begin{theorem}\label{invsem_th_ame2}
      Let $S$ be a countable and discrete inverse semigroup with identity $1 \in S$ and let $\alpha \colon S \to \mathcal I (X)$ be a representation. Then the following conditions are equivalent:
      \begin{enumerate}
        \item \label{invsem_th_ame2ame} $X$ is $S$-amenable.
        \item \label{invsem_th_ame2par} $D_e$ is not $S$-paradoxical for any $e \in E\left(S\right)$.
        \item \label{invsem_th_ame2atr} $\mathcal{R}_X$ has a trace $\phi$ such that $\phi\left(V_e\right) = 1$ for all $e \in E\left(S\right)$.
        \item \label{invsem_th_ame2inf} No projection $V_e \in \mathcal{R}_X$ is properly infinite.
      \end{enumerate}
    \end{theorem}

    In the appendix to \cite{H87}, Rosenberg showed that any countable discrete group $G$ with a quasidiagonal left regular representation is
amenable (see \cite{BO08}). This result implies that if $C^*_r(G)$ is quasidiagonal then $G$ is amenable (cf., \cite[Corollary~7.1.17]{BO08}). That
the reverse implication also holds was recently shown in \cite[Corollary~C]{TWW15}. Quasidiagonality of a C*-algebra can be defined in terms of a net
of unital completely positive (u.c.p.) maps (see, for example, \cite[Definition~7.1.1]{BO08} and \cite{V91,V93}):
    \begin{definition}\label{def:abstractqd}
      A unital separable C*-algebra $\mathcal{A}$ is called {\em quasidiagonal} if there exists a sequence of u.c.p. maps $\varphi _n \colon \mathcal{A}\to M_{k_n}(\mathbb C)$ which is both asymptotically multiplicative, i.e., $\| \varphi_n (AB) -\varphi_n(A)\varphi_n (B) \| \to 0$, $A,B\in \mathcal{A}$, and asymptotically isometric, i.e., $\| A \| =\lim_{n\to \infty} \| \varphi_n (A) \| $, $A\in \mathcal{A}$.
    \end{definition}

    We conclude this article by showing that Rosenberg's implication still holds for some special class of inverse semigroups. We also prove that the reverse implication is false (the so called \textit{Rosenberg conjecture} in the case of groups, see \cite{BO08,TWW15}). Recall that the \textit{reduced C$^*$-algebra of an inverse semigroup} is the C$^*$-algebra generated by the left regular representation:
    \begin{equation}
      C_r^*\left(X\right) := C^*\left(\left\{V_s\right\}_{s \in S}\right) \subset \mathcal{R}_X. \nonumber
    \end{equation}
    Note that the following theorem can only be true for the reduced C$^*$-algebras (either in this context or for groups), since the uniform Roe algebras $\mathcal{R}_X$ (and $\mathcal{R}_G$) are almost never finite, and thus almost never quasi-diagonal (see~\cite[Proposition~7.1.15]{BO08}).
    \begin{theorem}\label{invsem_thm_rosenberg}
      Let $S$ be a countable and discrete inverse semigroup with identity $1 \in S$, and let $\alpha \colon S\to \mathcal I (X)$ be a representation. Then:
      \begin{enumerate}
        \item \label{rosenberg_dommeas} If $C^*_r(X)$ is quasidiagonal then $X$ is $S$-domain measurable.
        \item \label{rosenberg_ame}     If $C^*_r(X)$ is quasidiagonal and $S$ has a minimal projection then $X$ is $S$-amenable.
        \item \label{rosenberg_cntexm}  There are amenable inverse semigroups with and without minimal projections with non-quasidiagonal reduced C$^*$-algebras.
      \end{enumerate}
    \end{theorem}
    \begin{proof}
      As is customary in the literature, we will denote by $\text{tr}_{k_n}$ the normalized trace in $M_{k_n}\left(\mathbb{C}\right)$, while
$\text{Tr}$ will stand for the usual trace, i.e., $\text{tr}_{k_n}\left(\cdot\right) = \text{Tr}\left(\cdot\right)/k_n$.

      The proof of (\ref{rosenberg_dommeas}) is a particular case of Proposition~7.1.6 in \cite{BO08}. Indeed, let $\varphi_n \colon C_r^*\left(X\right) \rightarrow M_{k_n}\left(\mathbb{C}\right)$ be a sequence of u.c.p. maps that are asymptotically multiplicative and isometric. Then any cluster point $\tau$ of $\left\{\text{tr}_{k_n} \circ \varphi_n\right\}_{n \in \mathbb{N}}$ is an amenable trace of $C_r^*\left(X\right)$. Thus let $\Phi$ be any hypertrace extending $\tau$, and consider the measure $\mu\left(A\right) := \Phi\left(P_A\right)$. It follows from Eq.~(\ref{eq_trace_meas}) that $\mu$ is a domain measure.

      Let $e_0 \in E\left(S\right)$ be the minimal projection and consider $\varphi_n \colon C^*_r(X) \rightarrow
\mathbb{M}_{k_n}\left(\mathbb{C}\right)$ a sequence of unital completely positive (u.c.p.) maps 
      that are asymptotically multiplicative and asymptotically isometric (see Definition~\ref{def:abstractqd}). To prove (\ref{rosenberg_ame}) we
will construct a new sequence of u.c.p. maps $\phi_n$ that are asymptotically multiplicative, asymptotically isometric {\em and} with asymptotically
normalized trace in $V_{e_0}$. For this, recall that, by Lemma~\ref{invsem_lemma_mincomm}, $e_0$ commutes with every $s \in S$.

      Observe that $\varphi_n\left(V_{e_0}\right)$ is a positive matrix, whose norm is greater than $1 - \varepsilon_n$ and such that $\left|\left|\varphi_n\left(V_{e_0}\right) - \varphi_n\left(V_{e_0}\right) \varphi_n\left(V_{e_0}\right)\right|\right| < \varepsilon_n$ for some $\varepsilon_n  > 0$
      with $\varepsilon_n \rightarrow 0$ when $n \rightarrow \infty$. We will assume that $\varepsilon_n < 1/4$ for all $n$. 
      A routine exercise shows that the spectrum of $\varphi_n\left(V_{e_0}\right)$ is contained in $\left[0, \delta_n\right) \cup \left(1 - \delta_n, 1\right]$
      where $\delta_n = \frac{1}{2}(1- \sqrt{1-4\varepsilon_n})$.  Let $r_n$ be the number of eigenvalues of $\varphi_n\left(V_{e_0}\right)$ in $\left(1 - \delta_n, 1 \right]$, and note 
      that $r_n \geq 1$ for large $n \in \mathbb{N}$, since $\left|\left|\varphi_n\left(V_{e_0}\right)\right|\right| \geq 1 - \varepsilon_n$. 
      Let $W_n \subset \mathbb{C}^{k_n}$ be the subspace generated by the eigenvectors of $\varphi_n\left(V_{e_0}\right)$ of eigenvalues close to $1$. Finally, let $Q_n \colon \mathbb{C}^{k_n} \rightarrow \mathbb{C}^{r_n}$ 
      be a linear map such that $Q_n |_{W_n}$ is an isometry onto $\mathbb{C}^{r_n}$ and $Q |_{W_n^{\perp}} = 0$, i.e.,
      \begin{equation}
        Q_n \colon \mathbb{C}^{k_n} \rightarrow \mathbb{C}^{r_n}, \;\; Q_nQ_n^* = 1_{r_n} \;\; \text{and} \;\; Q_n^*Q_n = P_{W_n}. \nonumber
      \end{equation}
      For each $n \in \mathbb{N}$ let $m_n \in \mathbb{N}$ be large enough such that
      \begin{equation}\label{rosenberg_choice_m}
        \frac{m_n r_n}{k_n + m_n r_n} \left(1 - \delta_n\right) \geq 1 - 2 \; \delta_n.
      \end{equation}
      Consider the maps:
      \begin{equation}\label{rosenberg_eq_phi}
        \phi_n \colon C_r^*\left(X\right) \rightarrow M_{k_n + m_n r_n}\left(\mathbb{C}\right), \quad A \mapsto \varphi_n\left(A\right) \oplus \big( Q_n \varphi_n\left(A\right) Q_n^* \otimes 1_{m_n}\big).
      \end{equation}
      By construction it is clear that the maps $\phi_n$ are unital, completely positive and asymptotically isometric. They are also asymptotically multiplicative. For this first observe that, using the minimality of $e_0$ and Lemma~\ref{invsem_lemma_mincomm}, $V_{e_0}$ commutes with every $A \in C^*_r\left(X\right)$. Thus, by summing and substracting $\varphi_n\left(V_{e_0}\right) \varphi_n\left(A\right)$, $\varphi_n\left(A\right) \varphi_n\left(V_{e_0}\right)$ and $\varphi_n\left(V_{e_0}A\right)$, we have
      \begin{align}
        || Q_n^*Q_n \varphi_n\left(A\right) - \varphi_n\left(A\right) Q_n^*Q_n || & \leq 2 || \varphi_n\left(A\right) || \; || \varphi_n\left(V_{e_0}\right) - Q_n^*Q_n || + || \varphi_n\left(V_{e_0}\right) \varphi_n\left(A\right) - \varphi_n\left(V_{e_0}A\right) || \nonumber \\
        & \quad \quad \quad \quad \quad \quad \quad \quad \quad + || \varphi_n\left(A\right) \varphi_n\left(V_{e_0}\right) - \varphi_n\left(A V_{e_0}\right)|| \nonumber \\
        & \leq 2 || A || \; || \varphi_n\left(V_{e_0}\right) - Q_n^*Q_n || + || \varphi_n\left(V_{e_0}\right) \varphi_n\left(A\right) - \varphi_n\left(V_{e_0}A\right) || \nonumber \\
        & \quad \quad \quad \quad \quad \quad \quad \quad \quad + || \varphi_n\left(A\right) \varphi_n\left(V_{e_0}\right) - \varphi_n\left(A V_{e_0}\right)|| \xrightarrow{n \rightarrow \infty} 0. \nonumber
      \end{align}
      This asymptotic commutation gives the asymptotic multiplicativity of $\phi_n$ by a straightforward computation since, for any $A, B \in C_r^*\left(X\right)$
      \begin{align}
        ||\phi_n\left(AB\right) - \phi_n\left(A\right) \phi_n\left(B\right)|| & \leq ||\varphi_n\left(AB\right) - \varphi_n\left(A\right) \varphi_n\left(B\right)|| \nonumber \\
        & \quad \quad \quad + || Q_n \varphi_n\left(AB\right) Q_n^* - Q_n \varphi_n\left(A\right) Q_n^*Q_n \varphi_n\left(B\right) Q_n^*|| \nonumber \\
        & \leq 2 \, || \varphi_n\left(AB\right)  - \varphi_n\left(A\right) \varphi_n\left(B\right) || \nonumber \\
        & \quad \quad \quad + ||B|| \, || Q_n^* Q_n \varphi_n\left(A\right) - \varphi_n\left(A\right) Q_n^*Q_n || \xrightarrow{n \rightarrow \infty} 0.\nonumber
      \end{align}
      Furthermore, the maps $\phi_n$ have the desired asymptotically normalized trace property at $V_e$:
      \begin{align}
        1 \geq \text{tr}\left(\phi_n\left(V_{e_0}\right)\right) & = \text{tr}\left(\varphi_n\left(V_{e_0}\right)\right) + m_n \, \text{tr}\left(Q_n \varphi_n\left(V_{e_0}\right) Q_n^*\right) \geq \frac{m_n}{k_n + m_n r_n} \text{Tr}\left(Q_n \varphi_n\left(V_{e_0}\right)Q_n^*\right) \nonumber \\
        & = \frac{m_n}{k_n + m_n r_n} \text{Tr}\left(Q_n^* Q_n \varphi_n\left(V_{e_0}\right)\right) \geq \frac{m_n r_n}{k_n + m_n r_n} \left(1 - \delta_n\right) \geq 1 - 2 \; \delta_n \xrightarrow{n \rightarrow \infty} 1, \nonumber
      \end{align}
      where the last inequality follows from the choice of $m_n$, see Eq.~(\ref{rosenberg_choice_m}). By the discussion in (\ref{rosenberg_dommeas}) and Theorem~\ref{invsem_th_ame2}, any cluster point $\Phi$ of $\left\{\text{tr}_{k_n + m_nr_n} \circ \phi_n\right\}_{n \in \mathbb{N}}$ will give an amenable trace of $C_r^*\left(X\right)$ normalized at $V_{e_0}$. Indeed, observe that
      \begin{equation}
        \mu\left(D_{e_0}\right) = \Phi\left(V_{e_0}\right) = \lim_{n \rightarrow \infty} \text{tr}_{k_n + m_n r_n}\left(\phi_n\left(V_{e_0}\right)\right) = 1. \nonumber
      \end{equation}
      We conclude that $\mu$ is a domain-measure localized at $D_{e_0}$ and, by Proposition~\ref{pro:invsem_cor_min} and Theorem~\ref{invsem_th_ame2}, $X$ is then $S$-amenable.

      Finally, for (\ref{rosenberg_cntexm}), consider the bicyclic inverse monoid $S = \langle a, a^* \mid a^*a = 1\rangle$. It is routine to show that $E(S) = \left\{1, aa^*, a^2(a^*)^2, \dots \right\}$ and, thus, $S$ has no minimal projection. Moreover, it is amenable (see~\cite[pp.~311]{DN78}) and $C^*_r\left(S\right)$ is not quasidiagonal (since it is not even finite, see~\cite[Proposition~7.1.15]{BO08}).
      
      For the case when $S$ does have a minimal projection, consider the semigroup $T = \mathbb{F}_2 \sqcup \left\{0\right\}$, where $0$ is a zero element. Since $T$ has a zero element it follows that $T$ is amenable (take $\mu$ atomic with total mass at $0$) and has a minimal projection (namely $0$). Furthermore, note that $C_r^*\left(S\right)$ is not quasidiagonal, since it contains the reduced group C*-algebra of $\mathbb{F}_2$, which is not quasidiagonal (see~\cite[Proposition~7.1.10]{BO08} and~\cite[Corollary~C]{TWW15}).
    \end{proof}

    \begin{remark}
      The condition of the existence of a minimal projection $e_0$ in Theorem~\ref{invsem_thm_rosenberg}(\ref{rosenberg_ame}) is assumed to guarantee
that it commutes with every element $s \in S$. In particular, the proof of (\ref{rosenberg_ame}) in Theorem~\ref{invsem_thm_rosenberg} can also be
applied to the more general setting of a unital and separable C$^*$-algebra $\mathcal{A}$ and a projection $P \in \mathcal{A}$, with $P$ commuting
with every $A \in \mathcal{A}$. Indeed, suppose that $\mathcal{A}$ is also quasidiagonal. Then, by the same argument, the construction of
Eq.~(\ref{rosenberg_eq_phi}) gives a quasidiagonal approximation of $\mathcal{A}$ with asymptotic normalized trace at $P$. In both cases, however, the
condition that $AP = PA$ for every $A \in \mathcal{A}$ is essential.
    \end{remark}
    
We conclude with some natural questions. Suppose $\alpha \colon S \rightarrow \mathcal{I}(X)$ is a representation of an inverse semigroup $S$ on a set
$X$.
    \begin{enumerate}
      \item[\textbf{Q1:}] Suppose that $C_r^*\left(X\right)$ is quasidiagonal. Is $X$ then $S$-amenable?
      \item[\textbf{Q2:}] Is there a stronger notion than the $S$-amenability of $X$ that guarantees the quasidiagonality of
$C_r^*\left(X\right)$?      
    \end{enumerate}

  %----------------------------------------------------------------------------------------
  % BIBLIOGRAPHY: format \bibitem{label} S.~Name, \textit{Title}, Journal. \textbf{volume} (year) page--page.
  %----------------------------------------------------------------------------------------
  \providecommand{\bysame}{\leavevmode\hbox to3em{\hrulefill}\thinspace}
  
\end{document}